\documentclass[12pt]{article}
\usepackage{amssymb}
\usepackage{verbatim}
\usepackage{array}
\usepackage{latexsym}
\usepackage{enumerate}
\usepackage{amsmath}
\usepackage{amsfonts}
\usepackage{amsthm}
\usepackage{graphicx,epsfig}
\usepackage[english]{babel}

\title{{\bf{3-nets realizing a group in a projective plane}}}
\date{}

\author{G.~Korchm\'aros, G.~P.~Nagy and N.~Pace.}

\newtheorem{theorem}{Theorem}[section]
\newtheorem{proposition}[theorem]{Proposition}
\newtheorem{lemma}[theorem]{Lemma}

{\theoremstyle{definition}
\newtheorem*{definition*}{Definition}
\newtheorem{example}[theorem]{Example}
\newtheorem{rem}[theorem]{Remark}
\newtheorem*{proposition*}{Proposition}
\newtheorem*{corollary*}{Corollary}
\newtheorem*{lemma*}{Lemma}

\def\cA{\mathcal A}
\def\cB{\mathcal B}
\def\cC{\mathcal C}
\def\cD{\mathcal D}

\def\cF{\mathcal F}
\def\cG{\mathcal G}

\def\cQ{\mathcal Q}
\def\cO{\mathcal O}
\def\cP{\mathcal P}

\def\cU{\mathcal U}

\def\cZ{\mathcal Z}
\begin{document}
\maketitle
\begin{abstract}
In a projective plane $PG(2,\mathbb K)$ defined over an algebraically closed field $\mathbb K$ of characteristic $0$, we give a complete classification of $3$-nets realizing a finite group. An infinite family, due to Yuzvinsky \cite{ys2004}, arises from plane cubics and comprises $3$-nets realizing cyclic and direct products of two cyclic groups. Another known infinite family, due to Pereira and Yuzvinsky \cite{per}, comprises $3$-nets realizing dihedral groups. We prove that there is no further infinite family. Urz\'ua's $3$-nets \cite{urzua2009} realizing the quaternion group of order $8$ are the unique sporadic examples.

If $p$ is larger than the order of the group, the above classification holds true in characteristic $p>0$ apart from three possible exceptions $\rm{Alt}_4$, $\rm{Sym}_4$ and $\rm{Alt}_5$.
\end{abstract}

\section{Introduction}
\label{problem}
In a projective plane a {\em{$3$-net}} consists of three pairwise disjoint classes of lines such that every point incident with two lines from distinct classes is incident with exactly one line from each of the three classes. If one of the classes has finite size, say $n$, then the other two classes also have size $n$, called the {\em{order}} of the $3$-net.

The notion of $3$-net comes from classical Differential geometry via the combinatorial abstraction of the notion of a $3$-web.
There is a long history about finite $3$-nets in Combinatorics related to affine planes, latin squares, loops and strictly transitive permutation sets. In this paper we are dealt with $3$-nets in a projective plane $PG(2,\mathbb K)$ over an algebraically closed field $\mathbb K$ which are coordinatized by a group. Such a $3$-net, with line classes $\cA,\cB,\cC$ and coordinatizing group $G=(G,\cdot)$, is equivalently defined by a triple of bijective maps from $G$ to $(\cA,\cB,\cC)$, say $$\alpha:\,G\to \cA,\,\beta:\,G\to \cB,\,\gamma:\,G\to \cC$$ such that $a\cdot b=c$ if and only if $\alpha(a),\beta(b),\gamma(c)$ are three concurrent lines in $PG(2,\mathbb{K})$, for any $a,b,c \in G$. If this is the case, the $3$-net in $PG(2,\mathbb{K})$ is said to {\em{realize}} the group $G$.
In recent years, finite $3$-nets realizing a group in the complex plane have been investigated in connection with complex line arrangements and Resonance theory see \cite{fy2007,miq,per,ys2004,ys2007}.

In the present paper, combinatorial methods are used to investigate finite $3$-nets realizing a group. Since key examples, such as algebraic $3$-nets and tetrahedron type $3$-nets, arise naturally in the dual plane of $PG(2,\mathbb{K})$, it is convenient to work with the dual concept of a $3$-net.

Formally, a {\em{dual $3$-net}} of order $n$ in $PG(2,\mathbb{K})$ consists of a triple $(\Lambda_1,\Lambda_2,\Lambda_3)$ with $\Lambda_1,\Lambda_2,\Lambda_3$ pairwise disjoint point-sets of size $n$, called {\em{components}}, such that every line meeting two distinct components meets each component in precisely one point. A dual $3$-net $(\Lambda_1,\Lambda_2,\Lambda_3)$ realizing  a group is {\em{algebraic}} if its points lie on a plane cubic, and is of {\em{tetrahedron type}} if its components lie on the six sides (diagonals) of a non-degenerate quadrangle such a way that $\Lambda_i=\Delta_i\cup \Gamma_i$ with $\Delta_i$ and $\Gamma_i$ lying on opposite sides, for $i=1,2,3$.

The goal of this paper is to prove the following classification theorem.
\begin{theorem}
\label{mainteo} In the projective plane $PG(2,\mathbb{K})$ defined over an algebraically closed field $\mathbb{K}$ of characteristic $p\geq 0$, let $(\Lambda_1,\Lambda_2,\Lambda_3)$ be a dual $3$-net of order $n\geq 4$ which realizes a group $G$. If either $p=0$ or $p>n$ then one of the following holds.
\begin{itemize}
\item[\rm{(I)}] $G$ is either cyclic or the direct product of two cyclic groups, and $(\Lambda_1, \Lambda_2, \Lambda_3)$ is algebraic.
\item[\rm{(II)}] $G$ is dihedral and $(\Lambda_1,\Lambda_2,\Lambda_3)$ is of tetrahedron type.
\item[\rm{(III)}] $G$ is the quaternion group of order $8$.
\item[\rm{(IV)}] $G$ has order $12$ and is isomorphic to $\rm{Alt}_4$.
\item[\rm{(V)}] $G$ has order $24$ and is isomorphic to $\rm{Sym}_4$.
\item[\rm{(VI)}] $G$ has order $60$ and is isomorphic to $\rm{Alt}_5$.
\end{itemize}
\end{theorem}
A computer aided exhaustive search shows that if $p=0$ then (IV) (and hence (V), (VI)) does not occur.

Theorem \ref{mainteo} shows that every realizable finite group can act in $PG(2,\mathbb{K})$ as a projectivity group. This confirms Yuzvinsky's conjecture for $p=0$.

The proof of Theorem \ref{mainteo} uses some previous results due to
Yuzvinsky \cite{ys2007}, Urz\'ua \cite{urzua2009}, and  Blokhuis, Korchm\'aros and Mazzocca \cite{bkm}.

Our notation and terminology are standard, see \cite{hp1973}. In view of Theorem \ref{mainteo}, $\mathbb{K}$ denotes an algebraically closed field of characteristic $p$ where either $p=0$ or $p\geq 5$, and any dual $3$-net in the present paper is supposed to be have order $n$ with $n<p$ whenever $p>0$.

\section{Some useful results on plane cubics}
\label{excubic}
A nice infinite family of dual $3$-nets realizing a cyclic group arises from plane cubics in $PG(2,\mathbb K)$;
see \cite{ys2004}. The key idea is to use
the well known abelian group defined on the points of an irreducible plane cubic, recalled here in the following two propositions.
\begin{proposition}{\rm{\cite[Theorem 6.104]{hkt}}}
\label{groupcubic} A non-singular plane cubic ${\cF}$ can be equipped with an additive group $(\cF,+)$ on the set of all its points. If an inflection point $P_0$ of ${\cF}$ is chosen to be the identity $0$, then three distinct points $P,Q,R\in \cF$ are collinear if and only if $P+Q+R=0$. For a prime number $d\neq p$, the subgroup of $(\cF,+)$ consisting of all elements $g$ with $dg=0$ is isomorphic to $C_d\times C_d$ while for $d=p$ it is either trivial or isomorphic to $C_p$ according as $\cF$ is supersingular or not.
\end{proposition}
\begin{figure}[ht]
\begin{center}
\begin{picture}(1,1)
\put(28,131){$P_{0}$}
\put(95,117){$P\oplus Q$}
\put(182,97){$R$}
\put(88,72){$Q$}
\put(7,48){$P$}
\put(44,124.5){\circle*{3}}
\put(94,112.5){\circle*{3}}
\put(176,93){\circle*{3}}
\put(99,70){\circle*{3}}
\put(15.5,45.5){\circle*{3}}
\end{picture}
\scalebox{.5}{\includegraphics{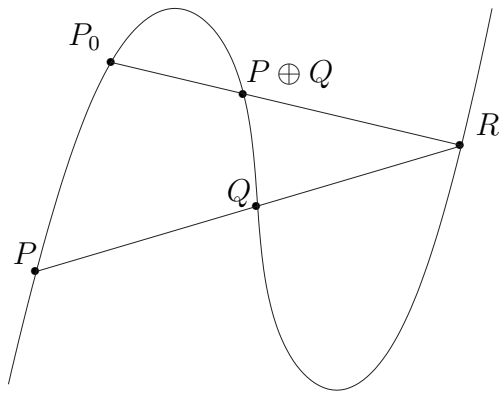}}
\end{center}
\caption{Abelian group law on an elliptic curve}
\label{fig112}
\end{figure}
\begin{proposition}{\rm{\cite[Proposition 5.6, (1)]{ys2004}}}.
\label{groupcubic1}
Let $\cF$ be an irreducible singular plane cubic  with its unique singular point $U$, and
define the operation $+$ on $\cF\setminus \{U\}$ in exactly the same way as on a non-singular plane cubic. Then $(\cF,+)$ is an abelian group isomorphic to the additive group of \,$\mathbb K$, or the multiplicative group of \,$\mathbb K$, according as $P$ is a cusp or a node.
\end{proposition}

If $P$ is a non-singular and non-inflection point of $\cF$ then the tangent to $\cF$ at $P$ meets $\cF$ a point $P'$ other than $P$, and $P'$ is the {\em{tangential}} point of $P$.
Every inflection point of a non-singular cubic $\cF$ is the center of an involutory homology preserving $\cF$.
A classical {\em{Lame configuration}} consists of two triples of distinct lines in
$PG(2,\mathbb{K})$, say $\ell_1,
\ell_2,\ell_3$ and $r_1,r_2,r_3$, such that no line from one triple passes through the
common point of two lines from the other triple. For $1\leq j,k\leq 3$,
let $R_{jk}$ denote the common point of the lines $\ell_j$ and $r_k$.
There are nine such common points, and they are called the points of the Lame configuration.
\begin{proposition}{\rm{Lame's Theorem.}}\,If eight points from a Lame configuration
lie on a plane cubic then the ninth also does.
\end{proposition}

\section{3-nets, quasigroups and loops}
\label{prelim}
A latin square of order $n$ is a table with $n$ rows and $n$ columns which has $n^2$ entries with $n$ different elements none of them occurring twice within any row or column. If $(L,*)$ is a quasigroup of order $n$ then its multiplicative table, also called Cayley table, is a latin square of order $n$, and the converse also holds.

For two integers $k,n$ both bigger than $1$, let $(G,\cdot)$ be a group of order $kn$ containing a normal subgroup $(H,\cdot)$ of order $n$. Let $\cG$ be a Cayley table of $(G,\cdot)$. Obviously, the rows and the columns representing the elements of $(H,\cdot)$ in $\cG$ form a latin square which is a Cayley table for  $(H,\cdot)$. From $\cG$, we may extract $k^2-1$  more latin squares using the cosets of $H$ in $G$. In fact, for any two such cosets $H_1$ and $H_2$, a latin square $H_{1,2}$ is obtained by taking as rows (respectively columns) the elements of $H_1$ (respectively $H_2$).

\begin{proposition}
\label{group} The latin square $H_{1,2}$ is a Cayley table for a quasigroup isotopic to the group $H$.
\end{proposition}
\begin{proof} Fix an element $t_1\in H_1$. In $H_{1,2}$, label the row representing the
element $h_1\in H_1$ with $h_1'\in H$ where $h_1=t_1\cdot h_1'$. Similarly, for a fixed
element $t_2\in H_2$, label the column representing the element $h_2\in H_2$ with $h_2'\in
H$ where $h_2=h_2'\cdot t_2$. The entries in $H_{1,2}$ come from the coset $H_1\cdot H_2$.
Now, label the entry $h_3$ in $H_1\cdot H_2$ with the element $h_3'\in H$ where
$h_3=t_1 \cdot h_3'\cdot t_2$.
Doing so, $H_{1,2}$ becomes a Cayley table for the subgroup $(H,\cdot)$, whence the assertion follows.
\end{proof}
In terms of a dual $3$-net, the relationship between $3$-nets and quasigroups can be described as follows. Let $(L,\cdot)$ be a loop arising from an embeddable $3$-net, and consider its dual $3$-net with its components $\Lambda_1,\Lambda_2,\Lambda_3$. For $i=1,2,3$, the points in $\Lambda_i$ are bijectively labeled by the elements of $L$. Let $(A_1,A_2,A_3)$ with $A_i\in \Lambda_i$ denote the the triple of the points corresponding to the element  $a\in L$. With this notation, $a\cdot b=c$ holds in $L$ if and only if the points $A_1,B_2$ and $C_3$ are collinear. In this way, points in $\Lambda_3$ are {\em{naturally labeled}} when $a \cdot b$ is the label of $C_3$.
 Let $(E_1,E_2,E_3)$ be the triple for the unit element $e$ of $L$. {}From $e\cdot e=e$, the points $E_1,E_2$ and $E_3$ are collinear. Since $a\cdot a=a$ only holds for $a=e$, the points $A_1,A_2,A_3$ are the vertices of a (non-degenerate) triangle whenever $a\neq e$.  Furthermore, from $e\cdot a=a$, the points $E_1,A_2$ and $A_3$ are collinear; similarly, $a\cdot e=a$ yields that the points $A_1,E_2,$ and $A_3$ are collinear. However, the points $A_1,A_2$ and $E_3$ form a triangle in general; they are collinear if and only if
$a\cdot a=e$, i.e. $a$ is an involution of $L$.

In some cases, it is useful to relabel the points of $\Lambda_3$ replacing the above bijection $A_3\to a$ from $\Lambda_3$ to $L$ by the bijection $A_3\to a'$ where $a'$ is the inverse of $a$ in $(L,\cdot)$. Doing so, three points  $A_1,B_2,C_3$ with $A_1\in \Lambda_1,\,B_2\in \Lambda_2,\,C_3  \in\Lambda_3$  are collinear if and only if $a\cdot b\cdot c=e$ with $e$ being the unit element in $(L,\cdot)$.
This new bijective labeling will be called a {\em{collinear relabeling}} with respect to $\Lambda_3$.

In this paper we are interested in $3$-nets of $PG(2,\mathbb{K})$ which are coordinatized by a group $G$. If this is the case, we say that the $3$-net realizes the group $G$. In terms of dual $3$-nets where $\Lambda_1,\Lambda_2,\Lambda_3$ are the three components, the meaning of this condition is as follows: There exists a triple of bijective maps from $G$ to $(\Lambda_1,\Lambda_2,\Lambda_3)$, say $$\alpha:\,G\to \Lambda_1,\,\beta:\,G\to \Lambda_2,\,\gamma:\,G\to \Lambda_3$$ such that $a\cdot b=c$ if and only if $\alpha(a),\beta(b),\gamma(c)$ are three collinear points, for any $a,b,c \in G$.

Let $(\Lambda_1,\Lambda_2,\Lambda_3)$ be a dual $3$-net that realizes a group $(G,\cdot)$ of order $kn$ containing a subgroup $(H,\cdot)$ of order $n$. Then the left cosets of $H$ provide a partition of each component $\Lambda_i$ into $k$ subsets. Such subsets are called left $H$-members and denoted by $\Gamma_i^{(1)},\ldots,\Gamma_i^{(k)}$, or simply $\Gamma_i$ when this does not cause confusion. The left translation map $\sigma_g:\, x\mapsto x+g$ preserves every left $H$-member.
The following lemma shows that every left $H$-member $\Gamma_1$ determines a dual $3$-subnet of $(\Lambda_1,\Lambda_2,\Lambda_3)$ that realizes $H$.
\begin{lemma}
\label{befana} Let $(\Lambda_1,\Lambda_2,\Lambda_3)$ be a dual $3$-net that realizes a group $(G,\cdot)$ of order $kn$ containing a subgroup $(H,\cdot)$ of order $n$. For any left coset $g\cdot H$ of $H$ in $G$, let $\Gamma_1=g\cdot H,\, \Gamma_2=H$ and $\Gamma_3=g\cdot H$. Then $(\Gamma_1,\Gamma_2,\Gamma_3)$ is a $3$-subnet of $(\Lambda_1,\Lambda_2,\Lambda_3)$ which realizes $H$.
\end{lemma}
\begin{proof} For any $h_1,h_2\in H$ we have that $(g\cdot h_1)\cdot h_2=g\cdot(h_1\cdot h_2)=g\cdot h$ with $h\in H$. Hence, any line joining a point of $\Gamma_1$ with a point of $\Gamma_2$ meets $\Gamma_3$.
\end{proof}
Similar results hold for right cosets of $H$. Therefore, for any right coset $H\cdot g$, the triple $(\Gamma_1,\Gamma_2,\Gamma_3)$ with $\Gamma_1=H,\,\Gamma_2=H\cdot g$ and $\Gamma_3=H\cdot g$ is a $3$-subnet of $(\Lambda_1,\Lambda_2,\Lambda_3)$ which realizes $H$.

The dual $3$-subnets $(\Gamma_1,\Gamma_2,\Gamma_3)$ introduced in Lemma \ref{befana} play a relevant role. When $g$ ranges over $G$, we obtain as many as $k$ such dual $3$-nets, each being called {\em{a dual $3$-net realizing the subgroup $H$ as a subgroup of $G$}}.

Obviously, left cosets and right cosets coincide if and only if $H$ is a normal subgroup of $G$, and if this is the case we may use the shorter term of coset.

Now assume that $H$ is a normal subgroup of $G$.
Take two $H$-members from different components, say $\Gamma_i$ and $\Gamma_j$  with $1\le i < j \le 3$. {}From Proposition \ref{group}, there exists  a member $\Gamma_m$ from the remaining component $\Lambda_m$, with $1\leq m \leq 3$ and $m\neq i,j$, such that $(\Gamma_1,\Gamma_2,\Gamma_3)$ is a dual 3-net of realizing $(H,\cdot)$. Doing so, we obtain $k^2$ dual $3$-subnets of $(\Lambda_1,\Lambda_2,\Lambda_3)$.  They are all the dual $3$-subnets of $(\Lambda_1,\Lambda_2,\Lambda_3)$ which realize the normal subgroup $(H,\cdot)$ as a subgroup of $(G,\cdot)$.
\begin{lemma}
\label{befanabis} Let $(\Lambda_1,\Lambda_2,\Lambda_3)$ be a dual $3$-net that realizes a group $(G,\cdot)$ of order $kn$ containing a normal subgroup $(H,\cdot)$ of order $n$. For any two cosets $g_1\cdot H$ and $g_2\cdot H$ of $H$ in $G$, let $\Gamma_1=g_1\cdot H,\, \Gamma_2=g_2\cdot H$ and $\Gamma_3=(g_1\cdot g_2)\cdot H$. Then $(\Gamma_1,\Gamma_2,\Gamma_3)$ is a $3$-subnet of $(\Lambda_1,\Lambda_2,\Lambda_3)$ which realizes $H$.
\end{lemma}
 If $g_1$ and $g_2$ range independently over $G$, we obtain as many as $k^2$ such dual $3$-nets, each being called {\em{a dual $3$-net realizing the normal subgroup $H$ as a subgroup of $G$}}.

\section{The infinite families of dual $3$-nets realizing a group}
A dual $3$-net $(\Lambda_1,\Lambda_2,\Lambda_3)$ with $n\geq 4$ is said to be {\em{algebraic}} if all its points lie on a (uniquely determined) plane cubic  $\cF$, called the {\em{associated}} plane cubic  of $(\Lambda_1,\Lambda_2,\Lambda_3)$.
Algebraic dual $3$-nets fall into three subfamilies according as the plane cubic splits into three lines, or in an irreducible conic and a line, or it is irreducible.

\subsection{Proper algebraic dual $3$-nets}
An algebraic dual $3$-net $(\Lambda_1,\Lambda_2,\Lambda_3)$ is said to be {\em{proper}} if its points lie on an irreducible plane cubic $\cF$.

\begin{proposition}
\label{algebraic1} Any proper algebraic dual 3-net $(\Lambda_1,\Lambda_2,\Lambda_3)$ realizes a group $M$. There is a subgroup $T\cong M$ in $(\cF,+)$ such that each component $\Lambda_i$ is a coset $T+g_i$ in $(\cF,+)$ where $g_1+g_2+g_3=0$.
\end{proposition}
\begin{proof}
We do some computation in $(\cF,+)$. Let $A_1,A_2,A_3 \in \Lambda_1$ three distinct points viewed as elements in $(\cF,+)$.  First we show that the solution of the equation in $(\cF,+)$
\begin{equation} \label{equation.in_A} A_1 - A_2 = X - A_3
\end{equation} belongs to $\Lambda_1$.
Let $C \in \Lambda_3$. From the definition of a dual $3$-net, there exist $B_i \in \Lambda_2$ such that
$A_i+B_i+C=0$ for $i=1,2,3$. Now choose $C_1\in \Lambda_3$ for which $ A_1+B_2+C_1 = 0 $, and then choose $A^* \in \Lambda_1$ for which $A^* + B_3 + C_1 =0$. Now,
\begin{equation}
\label{25022011}
\begin{array}{lll}
A^*- A_3 =  - B_3 - C_1 - (  -B_3  - C  )   =  C - C_1 \\
A_1- A_2  =  - B_2 - C_1 - (  -B_2  - C  ) =  C - C_1
\end{array}
\end{equation}
Therefore, $A^*$ is a solution of Equation (\ref{25022011}).

Now we are in a position to prove that $\Lambda_1$ is a coset of a subgroup of $(\cF,+)$. For $A_0\in \Lambda_1$, let $T_1=\{A-A_0|A\in \Lambda_1\}$. Since $(A_1-A_0)-(A_2-A_0)=A_1-A_2,$ Equation (\ref{25022011}) ensures the existence of $A^*\in \Lambda_1$ for which $A_1-A_2=A^*-A_0$ whenever $A_1,A_2\in \Lambda_1$. Hence $(A_1-A_0)-(A_2-A_0)\in T_1$. From this, $T_1$ is a subgroup of $(\cF,+)$, and therefore $\Lambda_1$ is a coset $T+g_1$ of $T_1$ in $(\cF,+)$.

Similarly, $\Lambda_2=T_2+g_2$ and $\Lambda_3=T_3+g_3$ with some subgroups $T_2,T_3$ of $(\cF,+)$ and elements $g_2,g_3\in (\cF,+)$. It remains to show that $T_1=T_2=T_3$. The line through the points $g_1$ and $g_2$
meets $\Lambda_3$ in a point $t^*+g_3$. Replacing $g_3$ with $g_3+t^*$ allows to assume that $g_1+g_2+g_3=0$. Then three points $g_i+t_i$ with $t_i\in T_i$ is collinear if and only if $t_1+t_2+t_3=0$. For $t_3=0$ this yields $t_2=-t_1$. Hence, every element of $T_2$ is in $T_1$, and the converse also holds. From this,  $T_1=T_2$. Now, $t_3=-t_1-t_2$ yields that $T_3=T_1$. Therefore $T=T_1=T_2=T_3$ and $\Lambda_i=T+g_i$ for $i=1,2,3$. This shows that $(\Lambda_1,\Lambda_2,\Lambda_3)$ realizes a group $M\cong T$.
\end{proof}
\subsection{Triangular dual 3-nets}
\label{triangular3net}
An algebraic dual $3$-net $(\Lambda_1,\Lambda_2,\Lambda_3)$ is {\em regular} if the components lie on three lines, and it is either of {\em{pencil type}} or {\em{triangular}} according as the three lines are either concurrent, or they are the sides of a triangle.
\begin{lemma}
\label{lem250220111} Every regular dual $3$-net of order $n$ is triangular.
\end{lemma}
\begin{proof} Assume that the components of a regular dual $3$-net $(\Lambda_1,\Lambda_2,\Lambda_3)$ lie on three concurrent lines. Using homogeneous coordinates in $PG(2,\mathbb{K})$, these lines are assumed to be those with equations $Y=0, X=0, X-Y=0$ respectively, so that the line of equation $Z=0$ meets each component. Therefore, the points in the components  may be labeled in such a way that
$$
\begin{array}{lll}
\Lambda_1=\{(1,0,\xi)|\xi\in L_1\},\,
\Lambda_2=\{(0,1,\eta)|\eta\in L_2\},\,
\Lambda_3=\{(1,1,\zeta)|\zeta\in L_3\},
\end{array}
$$
with $L_i$ subsets of $\mathbb{K}$ containing $0$. By a straightforward computation, three points $P=(1,0,\xi)$, $Q=(0,1,\eta)$, $R=(1,1,\zeta)$ are collinear if  and only if $\zeta=\xi+\eta$. Therefore, $L_1=L_2=L_3$ and  $(\Lambda_1,\Lambda_2,\Lambda_3)$ realizes a subgroup of the additive group of $\mathbb{K}$ of order $n$. Therefore $n$ is a power of $p$. But this contradicts the hypothesis $p>n$.
\end{proof}
For a triangular dual $3$-net, the (uniquely determined) triangle whose sides contain the components is called the {\em{associated}} triangle.
\begin{proposition}
\label{triangular} Every triangular dual $3$-net realizes a cyclic group isomorphic to a multiplicative group of $\mathbb{K}$.
\end{proposition}
\begin{proof} Using homogeneous coordinates in $PG(2,\mathbb{K})$, the vertices of the triangle are assumed to be the points $O=(0,0,1),\,X_\infty=(1,0,0),\,Y_\infty=(0,1,0).$ For $i=1,2,3$, let $\ell_i$ denote the fundamental line of equation $Y=0,\, X=0,\,Z=0$ respectively.
Therefore the points in the components lie on the fundamental lines and they may be labeled in such a way that
$$\Lambda_1=\{(\xi,0,1)|\xi\in L_1\},\,\Lambda_2=\{(0,\eta,1)|\eta\in L_2\},\,\Lambda_3=\{(1,-\zeta,0)|\zeta\in L_3\}$$
with $L_i$ subsets of $\mathbb{K}^*$ of a given size $n$. With this setting, three points $P=(\xi,0,1)$, $Q=(0,\eta,1)$, $R=(1,-\zeta,0)$
are collinear if and only if $\xi\zeta=\eta$.  With an appropriate choice of the unity point of the coordinate system, both $1\in L_1$ and $
1\in L_2$ may also be assumed. From $1\in L_1$, we have that $L_2=L_3$. This together with $1\in L_2$ imply that $L_1=L_2=L_3=L$. Since $1\in L$,
$L$ is a finite multiplicative subgroup of $\mathbb{K}$. In particular, $L$ is cyclic.
\end{proof}
\begin{rem}
\label{rem1} In the proof of Proposition \ref{triangular}, if the unity point of the coordinate system is arbitrarily chosen, the subsets $L_1,L_2$ and $L_3$ are not necessarily subgroups. Actually, they are cosets of (the unique)  multiplicative cyclic subgroup $H$, say $L_1=aH,\,L_2=bH$ and $L_3=cH$, with $ac=b$. Furthermore, since every $h\in H$ defines a projectivity $\sigma_h:\, x\mapsto hx$ of the projective line, and these projectivities form a group isomorphic to $H$, it turns out that $L_i$ is an orbit of a cyclic projectivity group of $\ell_i$ of order $n$, for $i=1,2,3$.
\end{rem}
\begin{proposition}
\label{trihom} Let $(\Lambda_1,\Lambda_2,\Lambda_3)$ be a triangular dual $3$-net. Then every point of $(\Lambda_1,\Lambda_2,\Lambda_3)$ is the center of a unique involutory homology which preserves $(\Lambda_1,\Lambda_2,\Lambda_3)$.
\end{proposition}
\begin{proof} The point $(\xi,0,1)$ is the center and the line through $Y_\infty$ and the point $(-\xi,0,1)$ and  is the axis of the involutory homology $\varphi_{\xi}$ associated to the  matrix
$$
\left(
  \begin{array}{ccc}
    0 & 0 & \xi^2 \\
    0 & -\xi & 0 \\
    1 & 0 & 0 \\
  \end{array}
\right).
$$
With the above notation, if $\xi\in aH$ then $h_\xi$ preserves $\Lambda_1$ while it sends any point in $\Lambda_2$ to a point in $\Lambda_3$,  and viceversa. Similarly, for $\eta\in bH$  and $\zeta\in cH$ where $\psi_\eta$ and $\theta_\zeta$ are the involutory homologies associated to the matrices
$$
\left(
  \begin{array}{ccc}
    -\eta & 0 & 0 \\
    0 & 0 & \eta^2 \\
    0 & 1 & 0 \\
  \end{array}
\right)\,\, {\rm{and}}\,\,
\left(
  \begin{array}{ccc}
   0 & 1 & 0 \\
    \zeta^2 & 0 & 0 \\
    0 & 0 & \zeta \\
   \end{array}
\right).
$$
\end{proof}
With the notation introduced in the proof of Proposition \ref{triangular}, let $\Phi_1=\{\varphi_{\xi}\varphi_{\xi'}|\xi,\xi'\in aH\}$ and $\Phi_2=\{\psi_{\eta}\psi_{\eta'}|\eta,\eta'\in bH\}$. Then both are cyclic groups isomorphic to $H$. A direct computation gives the following result.
\begin{proposition}
\label{trihomhom} $\Phi_1\cap \Phi_2$ is either trivial or has order $3$.
\end{proposition}
Some useful consequences are stated in the following proposition.
\begin{proposition}
\label{trihomhom+} Let $\Theta=\langle \Phi_1, \Phi_2\rangle$. Then
$$ |\Theta|= \left\{
\begin{array}{lll}
\,\,\,|H|^2,\,\,\,{\rm{when\,\, gcd.}}(3,|H|)=1;\\
\frac{1}{3}|H|^2, \,\, {\rm{when\,\, gcd.}}(3,|H|)=3.
\end{array}
\right.
$$ Furthermore, $\Theta$ fixes the vertices of the fundamental triangle, and no non-trivial element of $\Theta$ fixes a point outside the sides of the fundamental triangle.
\end{proposition}




We prove another useful result.
\begin{proposition}
\label{triangular1} If $(\Gamma_1,\Gamma_2,\Gamma_3)$ and $(\Sigma_1,\Sigma_2,\Sigma_3)$ are triangular dual 3-nets such that $\Gamma_1=\Sigma_1$, then the associated triangles share the vertices on their common side.
\end{proposition}
\begin{proof} From Remark \ref{rem1}, $\Gamma_1$ is the orbit of a cyclic projectivity group $H_1$ of the line $\ell$ containing $\Gamma_1$ while the two fixed points of $H_1$ on $\ell$, say $P_1$ and $P_2$, are vertices of the triangle containing $\Gamma_1,\Gamma_2,\Gamma_3$.

The same holds for $\Sigma_1$ with a cyclic projectivity group $H_2$, and fixed points $Q_1,Q_2$. From $\Gamma_1=\Sigma_1$, the projectivity group $H$ of the line $\ell$ generated by $H_1$ and $H_2$ preserves $\Gamma_1.$ Let $M$ be the projectivity group generated by $H_1$ and $H_2$.

Observe that $M$ is a finite group since it has an orbit of finite size $n\geq 3$. Clearly, $|M|\geq n$ and equality holds if and only if $H_1=H_2$. If this is the case, then $\{P_1,P_2\}=\{Q_1,Q_2\}$. Therefore, for the purpose of the proof, we may assume on the contrary that
$H_1\neq H_2$ and $|M|>n$.

Now, Dickson's classification of finite subgroups of $PGL(2,{\mathbb{K}})$ applies to $M$. From that classification, $M$ is one of the nine subgroups listed as ($(1),\ldots, (9)$ in  \cite[Theorem 1]{mv1983} where $e$ denotes the order of the stabilizer $M_P$ of a point $P$ in a short $M$-orbit, that is, an $M$-orbit of size smaller than $M$. Observe that such an $M$-orbit has size $|M|/e$. There exist a finitely many short $M$-orbits, and $\Sigma_1$ is one of them. It may be that an $M$-orbit is trivial as it consists of just one point.

Obviously, $M$ is neither cyclic nor dihedral as it contains two distinct cyclic subgroups of the same order $n\geq 3$.

Also, $M$ is not an elementary abelian $p$-group $E$ of rank $\geq 2$, otherwise we would have $|E|=|M|>n$ since the minimum size of a non-trivial $E$-orbit is $|E|$, see (2) in \cite[Theorem 1]{mv1983}.

{}From (5) in \cite[Theorem 1]{mv1983} with $p\neq 2,3$, the possible sizes of a short ${\rm{Alt_4}}$-orbit are $4,6$ each larger than $3$. On the other hand, ${\rm{Alt_4}}$ has no element of order larger than $3$. Therefore, $M\not\cong {\rm{Alt_4}}$ for $p\neq 2,3$.

Similarly, from (5) in \cite[Theorem 1]{mv1983} with $p\neq 2,3$, the possible sizes of a short ${\rm{Sym_4}}$-orbit are $6,8,12$ each larger than $4$. Since ${\rm{Sym_4}}$ has no element of order larger than $4$. Therefore, $M\not\cong {\rm{Sym_4}}$ for $p\neq 2,3$.

Again, from (6) in \cite[Theorem 1]{mv1983} with $p\neq 2,5$, the  possible sizes of a short ${\rm{Alt_5}}$-orbit are $10,12$ for $p=3$ while
$12,20,30$ for $p\neq 2,3,5$. Each size exceeds $5$. On the other hand ${\rm{Alt_5}}$ has no element of order larger than $5$. Therefore, $M\not\cong {\rm{Alt_5}}$ for $p\neq 2,5$.

The group $M$ might be isomorphic to a subgroup $L$ of order $qk$ with $k|(q-1)$ and $q=p^h$, $h\geq 1$. Here $L$ is the semidirect product of the unique (elementary abelian) Sylow $p$-subgroup of $L$ by a cyclic subgroup of order $k$. No element in $L$ has order larger than $k$ when $h>1$ and $p$ when $h=1$. From (7) in \cite[Theorem 1]{mv1983}, any non-trivial short $L$-orbit has size $q$. Therefore $M\cong L$ implies that $h=1$ and $n=p$. But this is inconsistent with the hypothesis $p>n$.

Finally, $M$ might be isomorphic to a subgroup $L$ such that either $L=PSL(2,q)$ or $L=PGL(2,q)$ with $q=p^h$, $h\geq 1$. No element in $L$ has order larger than $q+1$. From (7) and (8) in \cite[Theorem 1]{mv1983}, any short $L$-orbit has size either $q+1$ or $q(q-1)$. For $q\geq 3$, if $M\cong L$ occurs then $n=q+1\geq p+1$, a contradiction with the hypothesis $p>n$. For $q=2$, we have that $|L|=6$ which is smaller than $12$. Therefore $M\not\cong L$.

No possibility has arisen for $M$. Therefore $\{P_1,P_2\}=\{Q_1,Q_2\}.$
\end{proof}
\subsection{Conic-line type dual 3-nets}
An algebraic dual $3$-net $(\Lambda_1,\Lambda_2,\Lambda_3)$ is of {\em{conic-line type}} if two of its three components lie on an irreducible conic $\cC$ and the third one lies on a line $\ell$. All such $3$-nets realize groups and they can be described using subgroups of the projectivity group $PGL(2,\mathbb{K})$ of $\cC$. For this purpose, some basic results on subgroups and involutions in $PGL(2,\mathbb{K})$ are useful which essentially depend on the fact that every involution in $PGL(2,\mathbb{K})$ is a perspectivity whose center is a point outside $\cC$ and axis is the pole of the center with respect to the orthogonal polarity arising from $\cC$. We begin with an example.
\begin{example}
\label{exconicline}
Take any cyclic subgroup $C_n$ of $PGL(2,\mathbb{K})$ of order $n\geq 3$ with $n\neq p$ that preserves $\cC$. Let $D_n$ be the unique dihedral subgroup of $PGL(2,\mathbb{K})$ containing $C_n$. If $j$ is the (only) involution in $\cZ(D_n)$ and $\ell$ is its axis, then the centers of the other involutions in $D_n$ lie on $\ell$. We have $n$ involutions in $D_n$ other than $j$, and the set of the their centers is taken for $\Lambda_1$.
Take a $C_n$-orbit $\cO$ on $\cC$ such that the tangent to $\cC$ at any point in $\cO$ is disjoint from $\Lambda_1$; equivalently, the $D_n$-orbit $\cQ$ be larger than $\cO$. Then $\cQ$ is the union of $\cO$ together with another $C_n$-orbit. Take these two $C_n$-orbits for $\Lambda_2$ and $\Lambda_3$ respectively. Then $(\Lambda_1,\Lambda_2,\Lambda_3)$ is a conic-line dual $3$-net which realizes $C_n$. It may be observed that $\ell$ is a chord of $\cC$ and the multiplicative group of $\mathbb{K}$ has a subgroup of order $n$. 
\end{example}

The cyclic subgroups $C_n$ form a unique conjugacy class in $PGL(2,\mathbb{K})$.
For a cyclic subgroup $C_n$ of $PGL(2,\mathbb{K})$ of order $n$,
the above construction provides a unique example of a dual $3$-net realizing $C_n$.
\begin{proposition}
\label{cltype}  Up to projectivities, the conic-line dual $3$-nets of order $n$ are those described in Example  {\em{\ref{exconicline}}}.
\end{proposition}
\begin{proof} Let $(\Lambda_1,\Lambda_2,\Lambda_3)$ be a conic-line dual $3$-net of order $n$ such that $\Lambda_1\cup\Lambda_2$ is contained in an irreducible conic $\cC$ while $\Lambda_3$ lies on a line $\ell$. For every point $P\in \ell$ with $P\not\in \cC,$ let $\varphi_P$ be the (unique) involutory homology with center $P$ that preserves $\cC.$ The homologies $\varphi_P$ with $P$ ranging over $\Lambda_3$ generate a group $\Psi$ preserving $\cC$. Hence  $\Phi$ is isomorphic to a proper subgroup of $PGL(2,\mathbb{K})$. For any three involutory homologies with centers on a line $\ell$ which preserve an irreducible conic, their product is also an involutory homology with center on $\ell$. In our case, the latter center is also a point in $\Lambda_3$, as  $(\Lambda_1,\Lambda_2,\Lambda_3)$ is a dual $3$-net.
Therefore, $\Psi$ has order $2n$ and it contains an abelian subgroup $\Phi$ of index $2$. From the classification of finite subgroups of $PGL(2,\mathbb{K})$, $\Psi$ is a dihedral group. For  any point $R\in \Lambda_1$, the $\Psi$-orbit of $R$ is $\Lambda_1\cup \Lambda_2$ so that both $\Lambda_1$ and $\Lambda_2$ are $\Phi$-orbits. Here the tangent to $\cC$ at $R$ is disjoint from $\Lambda_3$ otherwise the $\Lambda_1\cup\Lambda_2$ would be of length smaller than $2n$. Therefore, $(\Lambda_1,\Lambda_2,\Lambda_3)$ is a conic-line dual $3$-net given in Example \ref{exconicline}. Projective equivalence follows from the fact that dihedral groups of $PGL(2,q)$ of the same order are conjugate in $PGL(2,q)$. Since $p>n$ is assumed when $p>0$, the order $\Phi$ is not $p$ and hence $\ell$ is not a tangent to $\cC$.
\end{proof}
A corollary of this is the following result.
\begin{proposition}
\label{cltype2} A conic-line dual $3$-net realizes a cyclic group $C_n$.
\end{proposition}
A final useful result is given in the following proposition.
\begin{proposition}
\label{cltype1} Let $(\Gamma_1,\Gamma_2,\Gamma_3)$ and $(\Delta_1,\Delta_2,\Delta_3)$ be two conic-line type dual 3-nets where $\Gamma_3$ lies on the line $\ell$ and $\Delta_3$ lies on the line $s$. If $n\geq 5$ and
$\Gamma_1=\Delta_1$ then $\ell=s$.
\end{proposition}
\begin{proof} Let $\cC$ be the unique irreducible conic through $\Gamma_1$. Replace $(\Lambda_1,\Lambda_2,\Lambda_3)$ by $(\Gamma_1,\Gamma_2,\Gamma_3)$ and subsequently by $(\Gamma_1,\Delta_2,\Delta_3$ in the proof of Proposition \ref{cltype}. Two cyclic groups $\Phi_1$ and $\Phi_2$ are obtained which shear a point-orbit $\Gamma_1$. The subgroup $\Phi$ of $PGL(2,\mathbb{K})$ generated by $\Phi_1$ and $\Phi_2$ also preserves $\Gamma_1$ and hence it is finite. Since $|\Gamma_1|\geq 5$, the classification of finite subgroups of $PGL(2,\mathbb{K})$ yields that $\Psi_1=\Psi_2$. Since any cyclic subgroup of $PGL(2,\mathbb{K})$ of order $\geq 5$ preserves only one non-tangent line, and $\Psi_1$ preserves $\ell$ while $\Psi_2$ does $s$, the assertion follows.
\end{proof}

\subsection{Tetrahedron type dual 3-nets}
In $PG(2,\mathbb{K})$, any non-degenerate quadrangle with its six sides (included the two diagonals)
 may be viewed as the projection of a tetrahedron of $PG(3,\mathbb{K})$. This suggests to call two sides of the quadrangle {\em{opposite}}, if they do not have any common vertex. With this definition, the six sides of the quadrangle  are partitioned into three couples of opposite sides. Let $(\Lambda_1,\Lambda_2,\Lambda_3)$  be a dual $3$-net of order $2n$ containing a dual $3$-subnet
\begin{equation}
\label{subnet} (\Gamma_1,\Gamma_2,\Gamma_3)
\end{equation}
of order $n$. Observe that $(\Lambda_1,\Lambda_2,\Lambda_3)$ contains three more dual $3$-subnets of order $n$. In fact,
for $\Delta_i=\Lambda_i\setminus \Gamma_i$, each of the triples below defines such a subnet:
\begin{equation}
\label{subnets} (\Gamma_1,\Delta_2,\Delta_3),\,(\Delta_1,\Gamma_2,\Delta_3),\,(\Delta_1,\Delta_2,\Gamma_3).
\end{equation}
Now, the dual 3-net $(\Lambda_1,\Lambda_2,\Lambda_3)$ is said to be {\em{tetrahedron-type}} if its components lie on the sides of a non-degenerate quadrangle such that $\Gamma_i$ and $\Delta_i$ are contained in opposite sides, for  $i=1,2,3$. Such a non-degenerate quadrangle is said to be {\em{associated}} to $(\Lambda_1,\Lambda_2,\Lambda_3)$.
Observe that each of the six sides of the quadrangle contains exactly one of the point-sets $\Gamma_i$ and $\Delta_i$. Moreover, each of the four dual $3$-subnets listed in (\ref{subnet}) and (\ref{subnets}) is triangular as each of its components, called a {\em{half-set}}, lies on a side of a triangle whose vertices are also vertices of the quadrangle. Therefore there are six half-sets in any dual $3$-net of tetrahedron type.
\begin{proposition}
\label{tetra} Any tetrahedron-type dual $3$-net realizes a dihedral group.
\end{proposition}
\begin{proof} The associated quadrangle is assumed to be the fundamental quadrangle of the homogeneous coordinate system in $PG(2,\mathbb{K})$, so that its vertices are $O,$ $X_\infty,Y_\infty$ together with the unity point $E=(1,1,1)$. By definition, the subnet (\ref{subnet}) is triangular. Without loss of generality,
$$\Gamma_1=\{(\xi,0,1)|\xi\in L_1\},\,\Gamma_2=\{(0,\eta,1)|\eta\in L_2\},\,\Gamma_3=\{(1,-\zeta,0)|\zeta\in L_3\}$$
where $L_1=aH,L_2=bH,L_3=cH$ are cosets of $H$ with $ac=b$, see Remark \ref{rem1}. We fix such triple $\{a,b,c\}$. Observe that
$(a,0,1)\in \Gamma_1$, $(0,b,1)\in \Gamma_2$ and $(1,-c,0)\in \Gamma_3$.
Furthermore,
$$\Delta_1=\{(1,\alpha,1)|\alpha\in M_1\},\,\Delta_2=\{(\beta,1,1)|\beta\in M_2\},\,\Delta_3=\{(1,1,\gamma)|\gamma\in M_3\}$$
with $M_1,M_2$ and $M_3$ subsets of ${\mathbb{K}}\setminus \{0,1\}$, each of size $n$.

The projectivity $\varphi_3$  with matrix representation
$$
\left(
  \begin{array}{ccc}
    1 & 0 & -1 \\
    0 & 1 & -1 \\
    0 & 0 & 1 \\
  \end{array}
\right)
$$
fixes the line $X_\infty Y_\infty$ pointwise and maps the subnet $(\Delta_1,\Delta_2,\Gamma_3)$ into the triangular dual $3$-net $(\Delta_1',\Delta_2',\Gamma_3)$ with associated triangle $OX_\infty Y_\infty$. Here
$$\Delta_1'=\{(0,\beta-1,1)|\beta\in M_1\},\,\,\Delta_2'=\{(\alpha-1,0,1)|\gamma\in M_3\}.$$
{}From Remark \ref{rem1}, the elements $\beta-1$ with $\beta$ ranging over $M_2$ form a coset of $H$, and the same holds for $M_1$ and $M_2$.
Therefore, there exist $d,e\in \mathbb{K}^*$ such that
$$\Delta_1'=\{(0,\sigma,1)|\sigma\in dH\},\,\,\Delta_2'=\{(\sigma,0,1)|\sigma\in eH\},$$
and $ec=d$. Applying $\varphi_3^{-1}$ gives
$$\Delta_1=\{(1,\sigma+1,1)|\sigma\in dH\},\,\,\Delta_2=\{(\sigma+1,1,1)|\sigma\in eH\}.$$

Repeating the above argument for the projectivity $\varphi_2$  with matrix representation
$$
\left(
  \begin{array}{ccc}
    1 & 0 & 0 \\
    -1 & 1 & 0 \\
    -1 & 0 & 1 \\
  \end{array}
\right)
$$
shows  that $\varphi_2$ maps the subnet $(\Gamma_1,\Delta_2,\Delta_3)$ into
a triangular dual $3$-net $(\Delta_1'',\Gamma_2,\Delta_3')$ with associated triangle $OX_\infty Y_\infty$. Here
$$\Delta_1''=\{(1,\alpha-1,0)|\alpha\in M_1\},\,\,\Delta_3'=\{(\frac{1}{\gamma-1},0,1)|\gamma\in M_3\}.$$ Again, from Remark \ref{rem1}, there exist $f\in \mathbb{K}^*$ such that
$$\Delta_3'=\{(\sigma,0,1)|\sigma\in fH\},$$
and $-df=b$. Applying $\varphi_2^{-1}$ gives
$$\Delta_3=\{(1,1,\frac{\sigma+1}{\sigma})|\sigma\in fH\}.$$

Now, the homogeneous coordinate system is modified a little, by changing $E$ to the new unity point $E'=(a^{-1},b^{-1},1)$. With this change,
$$
\begin{array}{lll}
\Gamma_1=\{(\xi,0,1)|\xi\in H\},\, \Gamma_2=\{(0,\eta,1)|\eta\in H\},\,
\Gamma_3=\{(1,-\zeta,0)|\zeta\in H\},\\
\Delta_1=\{(a^{-1},b^{-1}(\sigma+1),1)|\sigma\in dH\},\,\Delta_2=\{(a^{-1}(\sigma+1),b^{-1},1)\sigma\in eH\},\\
\Delta_3=\{(a^{-1},b^{-1},\frac{\sigma+1}{\sigma})|\sigma\in fH\}.
\end{array}
$$
The points $E_1=(1,0,1)$, $E_2=(0,1,1)$, and $E_3=(1,-1,1)$ are collinear, and the triple $(E_1,E_2,E_3)$ may be chosen to represent the
unity element of the loop $(L,*)$ arising from the above tetrahedron-type dual $3$-net $(\Lambda_1,\Lambda_2,\lambda_3)$. The points in $\Lambda_2$
may be labeled with $\eta\in H$ when the point $A_2=(0,\eta,1)$ is in $\Gamma_1$, and with $\sigma\in eH$ when $A_2=(a^{-1}(\sigma+1),b^{-1},1)$ is in $\Delta_1$. Since $\{E_1A_2,A_3\}$ and $\{A_1,E_2,E_3\}$ are both collinear triples of points, in the former case $$A_1=(\eta^{-1},1,0),\,\,A_2=(0,\eta,1),\,\,A_3=(1-\eta,0),$$
and $\eta*\eta=\eta^2$ in $L$ as the points $A_1,A_2$ and $(1,\eta^2,0)$ are collinear. This shows that $(\Gamma_1,\Gamma_2,\Gamma_3)$
realizes the cyclic group $H$. In the latter case,
$$A_1=(a^{-1},b^{-1}(\frac{b}{a}\sigma+1),1),\,A_2=(a^{-1}(\sigma+1),b^{-1},1),\,A_3=(a^{-1},b^{-1},\frac{-a/\sigma+1}{\-a/\sigma}),$$
and $\sigma*\sigma=1$ as $\{A_1,A_2,E_3\}$ is a collinear triple of points.
Therefore, $(\Lambda_1,\Lambda_2,\Lambda_3)$ realizes a dihedral group of order $2n$.


An alternative approach to the proof is to lift $(\Lambda_1,\Lambda_2,\Lambda_3)$ to the fundamental tetrahedron of $PG(3,\mathbb{K})$ so that the projection $\pi$ from the point $P_0=(1,1,1,1)$ on the plane $X_4=0$ returns $(\Lambda_1,\Lambda_2,\Lambda_3)$. For this purpose, it is enough to define the sets lying on the edges of the fundamental tetrahedron:
$$
\begin{array}{lll}
&\Gamma_1' = \{(\xi,0,1,0)|\xi\in L_1\},&\Gamma_2' = \{(0,\eta,1,0)|\eta\in L_2\},\\
&\Gamma_3' = \{(1,-\zeta,0,0)|\zeta\in L_3\}, &\Delta_1' = \{(0,\alpha-1,0,-1)|\alpha\in M_1\},\\
&\Delta_2' = \{(\beta-1,0,0,-1)|\beta\in M_2\}, &\Delta_3' = \{(0,0,\gamma-1,-1)|\gamma\in M_3\},
\end{array}
$$
and observe that $\pi(\Gamma_i')=\Gamma_i$ and $\pi(\Delta_i')=\Delta_i$ for $i=1,2,3$. Moreover, a triple $(P_1,P_2,P_3)$ of points with $P_i\in \Gamma_i\cup \Delta_i$ consists of collinear points if and only if if their projection does. Hence, $(\Gamma_1' \cup \Gamma_2', \Gamma_3' \cup \Delta_1', \Delta_2' \cup \Delta_3')$  can be viewed as a ``spatial'' dual $3$-net realizing the same group $H$. Clearly, $(\Gamma_1' \cup \Gamma_2', \Gamma_3' \cup \Delta_1', \Delta_2' \cup \Delta_3')$ is contained in the sides of the fundamental tetrahedron. We claim that these sides minus the vertices form an infinite spatial dual $3$-net realizing the dihedral group $2.\mathbb K^*$.

To prove this, parametrize the points as follows.
\begin{equation}
\begin{array}{rcl}
\Sigma_1 &=&\{ x_1=(x,0,1,0), (\varepsilon x)_1=(0,1,0,x) \mid x\in \mathbb K^*\}, \\
\Sigma_2 &=&\{ y_2=(1,y,0,0), (\varepsilon y)_2=(0,0,1,y) \mid y\in \mathbb K^*\}, \\
\Sigma_3 &=&\{ z_3=(0,-z,1,0),(\varepsilon z)_3=(1,0,0,-z) \mid z\in \mathbb K^*\}. \\
\end{array}
\end{equation}
Then,
\begin{eqnarray*}
x_1,y_2,z_3 \mbox{ are collinear } &\Leftrightarrow& z=xy, \\
(\varepsilon x)_1,y_2,(\varepsilon z)_3 \mbox{ are collinear } &\Leftrightarrow&
z=xy \Leftrightarrow \varepsilon z =(\varepsilon x)y,\\
x_1,(\varepsilon y)_2,(\varepsilon z)_3 \mbox{ are collinear } &\Leftrightarrow&
z=x^{-1}y \Leftrightarrow \varepsilon z =x(\varepsilon y),\\
(\varepsilon x)_1,(\varepsilon y)_2,z_3 \mbox{ are collinear } &\Leftrightarrow&
z=x^{-1}y \Leftrightarrow z =(\varepsilon x)(\varepsilon y).
\end{eqnarray*}

Thus, $(\Gamma_1' \cup \Gamma_2', \Gamma_3' \cup \Delta_1', \Delta_2' \cup \Delta_3')$ is a dual $3$-subnet of $(\Sigma_1,\Sigma_2,\Sigma_3)$ and $H$ is a subgroup of the dihedral group $2.\mathbb K^*$. As $H$ is not cyclic but it has a cyclic subgroup of index $2$, we conclude that $H$ is itself dihedral.
\end{proof}
\section{Classification of low order dual 3-nets}
\label{loworder}
An exhaustive computer aided search gives the following results.
\begin{proposition}
\label{clac8szeged} Any dual $3$-net realizing an abelian group of order $\leq 8$  is algebraic.
The dual of Urz\'ua's $3$-nets are the only dual $3$-net which realize the quaternion group of order $8$.
\end{proposition}
\begin{proposition}
\label{clac9szeged} Any dual $3$-net realizing an abelian group of order $9$  is algebraic.
\end{proposition}

\begin{proposition}
\label{clac12} If $p=0$, no dual $3$-net realizes ${\rm{Alt}_4}$.
\end{proposition}

\section{Characterizations of the infinite families}
\begin{proposition}
\label{th5.3} Every dual $3$-net realizing a cyclic group is algebraic.
\end{proposition}
\begin{proof} For $n=3$, we have that $3n=9$, and hence all points of the dual 3-net lie on a cubic. Therefore, $n\geq 4$ is assumed.

Let $(\Lambda_1,\Lambda_2,\Lambda_3)$ be a dual $3$-net of order $n$ which realizes the cyclic group $(L,*)$. Therefore, the points of each component
are labeled by $I_n$. After a collinear relabeling with respect to $\Lambda_3$, consider the configuration of the following nine points: $0,1,2$ from $\Lambda_1$, $0,1,2$ from $\Lambda_2$ and $n-1,n-2,n-3$ from $\Lambda_3$. For the seek of a clearer notation, the point with label $a$ in the component $\Lambda_m$ will be denoted by $a_m$.

The configuration presents six triples of collinear points, namely
\begin{itemize}
\item[(i)] $\{0_1,1_2,(n-1)_3\},\,\,\{1_1,2_2,(n-3)_3\},\,\,\{2_1,0_2,(n-2)_3\}$;
\item[(ii)] $\{0_1,2_2,(n-2)_3\},\,\,\{1_1,0_2,(n-1)_3\},\,\,\{2_1,1_2,(n-3)_3\}$;
\end{itemize}
Therefore, the corresponding lines form a Lame configuration. Furthermore, the three (pairwise distinct) lines determined by the two triples in (i) can be regarded as a totally reducible plane cubic, say $\cF_1$. Similarly, a totally reducible plane curve, say $\cF_2$, arises from the triples in (ii). Obviously, $\cF_1\neq \cF_2$. Therefore, the nine points of the above Lame configuration are the base points of the pencil generated by $\cF_1$ and $\cF_2$. Now, define the plane cubic $\cF$ to be the cubic  from the pencil which contains $3_1$.

Our next step is to show that $\cF$ also contains each of the points $(n-4)_3$ and $3_2$.
For this purpose, consider the following six triples of collinear points
\begin{itemize}
\item[(iii)] $\{1_1,2_2,(n-3)_3\},\,\,\{2_1,0_2,(n-2)_3\},\,\,\{3_1,1_2,(n-4)_3\}$;
\item[(iv)] $\{1_1,1_2,(n-2)_3\},\,\,\{2_1,2_2,(n-4)_3\},\,\,\{3_1,0_2,(n-3)_3\}$;
\end{itemize}
Again, the corresponding lines form a Lame configuration. Since eight of its points, namely $1_1,2_1,3_1,0_2,1_2,2_2,(n-2)_3,(n-3)_3$ lie on $\cF$,
Lame's theorem shows that $(n-4)_3$ also lies on $\cF$. To show that $3_2\in \cF$, we proceed similarly using the following six triples of collinear points
\begin{itemize}
\item[(v)] $\{0_1,3_2,(n-3)_3\},\,\,\{1_1,1_2,(n-2)_3\},\,\,\{2_1,2_2,(n-4)_3\}$;
\item[(vi)] $\{0_1,2_2,(n-2)_3\},\,\,\{1_1,3_2,(n-4)_3\},\,\,\{2_1,1_2,(n-3)_3\}$;
\end{itemize}
to define a Lame configuration that behaves as before, eight of its points, namely $0_1,1_1,2_1,1_2,2_2,(n-2)_3,(n-3)_3,(n-4)_3$ lie on $\cF$,
from Lame's theorem,  $3_2$ also lies on $\cF$.

This completes the proof for $n=4$. We assume that $n\geq 5$ and show that $(n-5)_3$ lies on $\cF$. Again, we use the above argument based on the
Lame configuration of the six lines arising from the following six triples of points:
\begin{itemize}
\item[(vii)] $\{1_1,3_2,(n-4)_3\},\,\,\{2_1,1_2,(n-3)_3\},\,\,\{3_1,2_2,(n-5)_3\}$;
\item[(viii)] $\{1_1,2_2,(n-3)_3\},\,\,\{2_1,3_2,(n-5)_3\},\,\,\{3_1,1_2,(n-4)_3\}$;
\end{itemize}
{}From the previous discussion, eight of these points lie on $\cF$. Lame's theorem yields that the ninth, namely $(n-5)_3$, also lies on $\cF$.
{}From this we infer that $4_1\in \cF$ also holds. To do this, we repeat the above argument for the Lame configuration arising from the six triples of points
\begin{itemize}
\item[(ix)] $\{2_1,2_2,(n-4)_3\},\,\,\{3_1,0_2,(n-3)_3\},\,\,\{4_1,1_2,(n-5)_3\}$;
\item[(x)] $\{2_1,1_2,(n-3)_3\},\,\,\{3_1,2_2,(n-5)_3\},\,\,\{4_1,0_2,(n-4)_3\}$;
\end{itemize}
Again, we see that eight of these points lie on $\cF$. Hence the ninth, namely $4_1$, also lies on $\cF$, by Lame's theorem.

Therefore, from the hypothesis that $\cF$ passes through the ten points $$0_1,1_1,2_1,3_1,0_2,1_2,2_2,(n-1)_3,(n-2)_3,(n-3)_3,$$
we have deduced that $\cF$ also passes through the ten points
$$1_1,2_1,3_1,4_1,1_2,2_2,3_2,(n-2)_3,(n-3)_3,(n-4)_3. $$
Comparing these two sets of ten points shows that the latter derives from the former shifting by $+1$ when the indices are $1$ and $2$, while by $-1$ in when the indices are $3$. Therefore, repeating the above argument $n-4$ times gives that all points in the dual $3$-net lie on $\cF$.
\end{proof}

\begin{proposition}{\rm{\cite[Theorem 5.4]{ys2004}}}
\label{th5.4} If an abelian group $G$ contains an element of order $\geq 10$ then every dual $3$-net realizing $G$ is algebraic.
\end{proposition}

\begin{proposition}
\label{noelem}
{\rm{\cite[Theorem 4.2]{ys2004}}} No dual $3$-net realizes an elementary abelian group of order $2^h$ with $h\geq 3$.
\end{proposition}
\begin{proposition}{\rm{\cite[Theorem 5.1]{bkm}}}
\label{bkmth5.1} Let $(\Lambda_1,\Lambda_2,\Lambda_3)$ be a dual $3$-net such that at least one component lies on a line. Then $(\Lambda_1,\Lambda_2,\Lambda_3)$ is either triangular or of conic-line type.
\end{proposition}
\begin{lemma}
\label{20july} Let $(\Gamma_1,\Gamma_2,\Gamma_3)$ be an algebraic dual $3$-net lying on a plane cubic $\cF$.
If $\cF$ is reducible, then $(\Gamma_1,\Gamma_2,\Gamma_3)$ is either triangle or of conic-line type, according as $\cF$ and splits into three lines or into a line and an irreducible conic.
\end{lemma}

\begin{proposition}
\label{tetradiedral} Every dual $3$-net realizing a dihedral group of order $2n$ with $n\ge 3$ is of tetrahedron type.
\end{proposition}
\begin{proof}
Let $(\Lambda_1 , \Lambda_2 , \Lambda_3 )$ be a dual $3$-net realizing a dihedral group $$D_n = \left< x,y \; | \;  x^2=y^n=1, \; yx = x y^{-1} \right>.$$
Labeling naturally the points in the components $\Lambda_i$ as indicated in Section \ref{prelim}, every $u\in D_n$ defines a triple of points $(u_1,u_2,u_3)$ where $u_i\in \Lambda_i$ for $i=1,2,3$, and viceversa. Doing so, three points $u_1\in \Lambda_1,\,v_2\in \Lambda_2,\,w_3\in\Lambda_3$ are collinear if and only if $uv=w$ holds in $D_n$.

Therefore, for $1 \leq i \leq n-2$, the triangle with vertices $x_2,(xy)_2 ,(xy^{-i})_3$ and that with vertices $(1)_3,y_3,(y^{-i})_2$ are in mutual perspective position from the point $x_1$.  For two distinct points $u_i$ and $v_j$ with $u_i\in\Lambda_i$ and $v_j\in \Lambda_j$ and $1\leq i,j\leq 3$, let $\overline{u_iv_j}$ denote the line through $u_i$ and $v_j$. From the Desargues theorem, the three diagonal points, that is, the points
\begin{eqnarray*}
U&=&\overline{(x)_2 (xy)_2} \cap \overline{(1)_3 (y)_3}, \\ (y^i)_1&=&\overline{(x)_2 (xy^{-i})_3} \cap \overline{ (y^{-i})_2 (1)_3}, \\ (y^{i+1})_1&=&\overline{(xy)_2 (xy^{-i})_3} \cap \overline{( y^{-i})_2 (y)_3},
\end{eqnarray*}
are collinear. Hence, a line $\ell_1$ contains each point $(1)_1, (y)_1 \ldots, (y^{n-1})_1$ in $\Lambda_1$, that is,
$$ (1)_1, (y)_1 \ldots, (y^{n-1})_1\in \ell_1.$$

There are some more useful Desargues configurations. Indeed, the pairs of triangles with vertices
\begin{eqnarray*}
(x)_2,(xy^{-1})_2,(y^{-i-1})_3     & \text{and} &  (xy)_3,(x)_3,(y^{-i})_2 ; \\
(y^i)_2, (y^{i+1})_2, (y^{i+1})_3  & \text{and} &   (x)_3 , (xy)_3 , (xy)_2;\\
(xy^i)_2,(xy^{i+1})_2,(y^i)_3      & \text{and} &  (x)_3,(xy)_3,(1)_2 ;\\
(1)_2,(y)_2, (x)_3                 & \text{and} &   (y^i)_3,(y^{i+1})_3, (xy^i)_2 ;\\
(x)_2,(xy)_2, (1)_3                & \text{and} &  (xy^i)_3,(xy^{i+1})_3, (y^i)_2
\end{eqnarray*}
are in mutual perspective position from the points $$(y^{-1})_1, (xy^{-i})_1,(y^i)_1, (y^i)_1 , (y^{-i})_1 , $$ respectively. Therefore, there exist five more lines $m_1, \ell_2, m_2, \ell_3, m_3$ such that \begin{eqnarray*}
 \{ (x)_1, (xy)_1, \ldots , (xy^{n-1})_1 \}   \subset  m_1, &
 \{ (1)_2, (y)_2, \ldots , ( y^{n-1})_2 \}   \subset  \ell_2, \\
\{ (x)_2, (xy)_2, \ldots , ( xy^{n-1})_2 \}   \subset  m_2 , &
\{ (1)_3, (y)_3, \ldots , ( y^{n-1})_3 \}   \subset  \ell_3, \\
\{ (x)_3, (xy)_3, \ldots , ( xy^{n-1})_3 \}   \subset  m_3 . &
\end{eqnarray*}
By Proposition \ref{triangular1}, the lines $\ell_1,\ldots, m_3$ are the sides of a nondegenarate quadrangle, which shows that the dual $3$-net $(\Lambda_1,\Lambda_2,\Lambda_3)$ is of tetrahedron type.
\end{proof}
\begin{rem} {}From Proposition \ref{tetradiedral}, the dual $3$-nets given in \cite[Section 6.2]{per} are of tetrahedron type.
\end{rem}
\begin{proposition}
\label{cortetra0} Let $G$ be a finite group containing a normal subgroup $H$ of order $n\geq 3$. Assume that $G$ can be realized by a dual $3$-net $(\Lambda_1,\Lambda_2,\Lambda_3)$ and that every dual $3$-subnet of $(\Lambda_1,\Lambda_2,\Lambda_3)$ realizing $H$ as a subgroup of $G$ is triangular.
Then $H$ is cyclic and $(\Lambda_1,\Lambda_2,\Lambda_3)$ is either triangular or of tetrahedron type.
\end{proposition}
\begin{proof} From Proposition \ref{triangular}, $H$ is cyclic.
Fix an $H$-member $\Gamma_1$ from $\Lambda_1$, and denote by $\ell_1$ the line containing $\Gamma_1$. Consider all the triangles which contain some dual $3$-net $(\Gamma_1,\Gamma_2^j,\Gamma_3^s)$ realizing $H$ as a subgroup of $G$. From Proposition \ref{triangular1}, these triangles have two common vertices, say $P$ and $Q$, lying on $\ell_1$. For the third vertex $R_j$ of the triangle containing $(\Gamma_1,\Gamma_2^j,\Gamma_3^s)$ there are two possibilities, namely either the side $PR_j$ contains $\Gamma_2^j$ and the side $QR_j$ contains $\Gamma_3^s$, or viceversa. Therefore, every $H$-member $\Gamma_2^j$ from $\Lambda_2$ (as well as every $H$-member $\Gamma_3^s$ from $\Lambda_3$) is contained in a line passing through $P$ or $Q$.

Now, replace $\Gamma_1$ by another $H$-orbit $\Gamma_1^i$ lying in $\Lambda_1$ and repeat the above argument. If $\ell_i$ is the line containing $\Gamma_1^i$ and $P_i,Q_i$ denote the vertices then again every $H$-member $\Gamma_2^j$ from $\Lambda_2$ (as well as every $H$-member $\Gamma_3^s$ from $\Lambda_3$) is contained in a line passing through $P_i$ or $Q_i$.

Assume that $\{P,Q\}\neq \{P_i,Q_i\}$. If one of the vertices arising from $\Gamma_1$, say $P$, coincides with one of the vertices, say $P_i$, arising  from $\Gamma_1^i$ then the line $QQ_i$ must contain either $\Gamma_2^j$ or $\Gamma_3^s$ from each  $(\Gamma_1,\Gamma_2^j,\Gamma_3^s)$. Therefore, the line $QQ_i$ must contain every $H$-member from $\Lambda_2$, or every $H$-member from $\Lambda_3$. Hence $\Lambda_2$ or $\Lambda_3$ lies on the line $QQ_i$. From Proposition \ref{bkmth5.1}, $(\Lambda_1,\Lambda_2,\Lambda_3)$ is either triangular or conic-line type. The latter case cannot actually occur as $\Lambda_1$ contains $\Gamma_1$ and hence it contains at least three collinear points.

Therefore  $\{P,Q\}\cap \{P_i,Q_i\}=\emptyset$ may be assumed. Then the $H$-members from $\Lambda_2$ and $\Lambda_3$ lie on four lines,
namely $PP_i,PQ_i,QP_i,QQ_i$. Observe that these lines may be assumed to be pairwise distinct, otherwise  $\Lambda_2$ (or $\Lambda_3$) is contained in a line, and again $(\Lambda_1,\Lambda_2,\Lambda_3)$ is triangular. Therefore, half of the $H$-members from $\Lambda_2$ lie on one of these four lines, say $PQ_i$, and half of them on $QP_i$. Similarly, each of the lines $PP_i$ and $QQ_i$ contain half from the $H$-members from $\Lambda_3$.

In the above argument, any $H$-member $\Gamma_2$ from $\Lambda_2$ may play the role of $\Gamma_1$. Therefore there exist two lines such that each $H$-member from $\Lambda_1$ lies on one or on other line. Actually, these two lines are $PQ$ and $P_iQ_i$ since each of them contains a $H$-member from $\Lambda_1$. In this case, $(\Lambda_1,\Lambda_2,\Lambda_3)$ is of tetrahedron type.
\end{proof}
Since a dihedral group of order $\geq 8$ has a unique cyclic subgroup of index $2$ and such a subgroup is characteristic, Propositions \ref{cortetra0} and \ref{tetra} have the following corollary.
\begin{proposition}
\label{cortetra} Let $G$ be a finite group of order $n\geq 12$ containing a normal dihedral subgroup $D$.
If $G$ is realized by a dual $3$-net then $G$ is itself dihedral.
\end{proposition}

\section{Dual $3$-nets preserved by projectivities}
\label{seccolli}
\begin{proposition}
\label{charinv} Let $(\Lambda_1,\Lambda_2,\Lambda_3)$ be a dual $3$-net of order $n\geq 4$ realizing a group $G$. If every point in $\Lambda_1$ is the center of an involutory homology which preserves $\Lambda_1$ while interchanges $\Lambda_2$ with $\Lambda_3$, then either $\Lambda_1$ is contained in a line, or $n=9$. In the latter case, $(\Lambda_1,\Lambda_2,\Lambda_3)$ lies on a non-singular cubic $\cF$ whose inflection points are the points in $\Lambda_1$.
\end{proposition}
\begin{proof} After labeling $(\Lambda_1,\Lambda_2,\Lambda_3)$ naturally, take an element $a\in G$ and denote by $\varphi_a$ the (unique) involutory homology of center $A_1$ which maps $\Lambda_2$ onto $\Lambda_3$. Obviously, $\varphi_a$ also maps $\Lambda_3$ onto $\Lambda_2$. Moreover,
$\varphi_a(X_2)=Y_3 \Longleftrightarrow a\cdot x=y,$ that is, $\varphi_a(X_2)=\varphi_{a'}(X_2') \Longleftrightarrow a\cdot x=a'\cdot x',$  where $G=(G,\cdot)$. Therefore,
\begin{equation}
\label{eq027febbr}
\varphi_{a'}\varphi_a(X_2)=X_2' \Longleftrightarrow ({a'}^{-1}\cdot a)\cdot x =x'.
\end{equation}
{}From this, for any $b\in G$ there exists $b'\in G$ such that
\begin{equation}
\label{eq27febbr}
\varphi_{a'}\varphi_{a}(X_2)=\varphi_{b'}\varphi_{b}(X_2)
\end{equation}
for every $X_2\in \Lambda_2$, equivalently, for every $x\in G$.

Let $\Phi$ be the the projectivity group generated by all products $\varphi_{a'}\varphi_a$ where both $a,a'$ range over $G$. Obviously, $\Phi$ leaves both $\Lambda_2$ and $\Lambda_3$ invariant. In particular, $\Phi$ induces a permutation group on $\Lambda_2$. We show that if $\mu\in \Phi$ fixes $\Lambda_2$ pointwise then $\mu$ is trivial. Since $n>3$, the projectivity $\mu$ has at least four fixed points in $PG(2,\mathbb{K})$. Therefore, $\mu$ is either trivial, or a homology. Assume that $\mu$ is non-trivial, and let $C$ be  the center and $c$ the axis of $\mu$. Take a line $\ell$ through $C$ that contains a point $P\in\Lambda_3$, and assume that $C$ is a point in $\Lambda_2$. Then $P$ is the unique common point of $\ell$ and $\Lambda_3$. Since $\mu$ preserves $\Lambda_2$,  $\mu$ must fix $P$. Therefore, $\mu$ fixes $\Lambda_3$ pointwise, and hence $\Lambda_3$ is contained in $c$. But then $\mu$ cannot fix
 any point in $\Lambda_2$ other than $C$ since the definition of a dual $3$-net implies that $c$ is disjoint from $\Lambda_2$. This contradiction means that $\mu$ is trivial, that is, $\Phi$ acts faithfully on $\Lambda_2$.

Therefore, (\ref{eq27febbr}) states that for any $a,a',b\in G$ there exists $b'\in G$ satisfying the equation  $\varphi_{a'}\varphi_{a}=\varphi_{b'}\varphi_{b}$. This yields that $\Phi$ is an abelian group of order $n$ acting on $\Lambda_2$ as a sharply transitive permutation group. Also, $$\Phi=\{\varphi_a\varphi_e\mid a\in G\}$$ where $e$ is the identity of $G$. Therefore, $\Phi\cong G$, and $G$ is abelian.

Let $\Psi$ be the projectivity group generated by $\Phi$ together with some $\varphi_a$ where $a\in G$. Then $|\Psi|=2n$ and $\Psi$ comprises the elements in $\Phi$ and the involutory homologies $\varphi_a$ with $a$ ranging over $G$. Obviously, $\Psi$ interchanges $\Lambda_2$ and $\Lambda_3$ while it leaves $\Lambda_1$ invariant acting on $\Lambda_1$ as a transitive permutation group.

 Two cases are investigated according as $\Phi$ contains a homology or does not. Observe that $\Phi$ contains no elation, since every elation has infinite order when $p=0$ while it's order is at least $p$ when $p>0$ but $p>n$ is assumed throughout the paper.

 In the former case, let $\rho\in\Phi$ be a homology with center $C\in \Lambda_1$ and axis $c$. Since $\rho$ commutes with every element in $\Phi$, the point $C$ is fixed by $\Phi$, and the line is preserved by $\Phi$. Assume that $C$ is also the center of $\phi_a$ with some $a\in G$. The group of homologies generated by $\phi_a$ and $\rho$ preserves every line through $C$ and it has order bigger than $2$. But then it cannot interchange $\Lambda_2$ with $\Lambda_3$. Therefore, the center of every $\phi_a$ with $a\in G$ lies on $c$. This shows that $\Lambda_1$ is contained in $c$.

In the case where $\Phi$ contains no homology, $\Phi$ has odd order and  $\delta\in\Phi$ has three fixed points which are the vertices  of a triangle $\Delta$. Since $\delta$ commutes with every element in $\Phi$, the triangle $\Delta$ is left invariant by $\Phi$.

If $\Phi$ fixes each vertex of $\delta$, then $\Phi$ must be cyclic otherwise $\Psi$ would contain a homology. Therefore $\Psi$ is a dihedral group, and we show that $\Lambda_1$ is contained in a line. For this purpose, take a generator $\rho=\varphi_a\varphi_b$ of $\Phi$, and consider the line $\ell$ through the centers of $\varphi_a$ and $\varphi_b$. Obviously, $\rho$ preserves $\ell$, and this holds true for every power of $\rho$. Hence $\Psi$ also preserves $\ell$. Since every $\varphi_c$ is conjugate to $\varphi_a$ under $\Psi$, this shows that the center of $\varphi_c$ must lie on $\ell$, as well. Therefore $\Lambda_1$ is contained in $\ell$.

We may assume that some $\rho\in\Phi$ acts on the vertices of $\Delta$ as a $3$-cycle. Let $\Delta'$ be the triangle whose vertices are the fixed points of $\rho$. Then $\rho^3=1$ since $\rho^3$ fixes not only the vertices of $\Delta'$ but also those of $\Delta'$. Therefore $\Phi=\langle \rho\rangle\times \Theta$ where $\Theta$ is the cyclic subgroup of $\Phi$ fixing each vertex of $\Delta$. A subgroup of $\Theta$ of index $\leq 3$ fixes each vertex of $\Delta'$, and hence it is trivial. Therefore, $|\Theta|=3$ and  $\Phi\cong C_3\times C_3$. This shows that $n=9$ and if $\Lambda_1$ is not contained in a line then the configuration of their points, that is the the centers of the homologies in $\Psi$, is isomorphic to $AG(2,3)$, the affine plane of order $3$. Such a configuration can also be viewed as  the set of the nine common inflection points of the non-singular plane cubics of a pencil $\mathcal P$, each cubic left invariant by $\Psi$. For a point $P_2\in \Lambda_2$, take that cubic $\cF$ in $\mathcal P$ that contains $P_2$. Since the orbit of $P_2$ under the action of $\Psi$ consists of the points in $\Lambda_2\cup \Lambda_3$, it follows that $\cF$ contains each point of $(\Lambda_1,\Lambda_,\Lambda_3)$.
\end{proof}

A corollary of Proposition \ref{charinv} is the following result.
\begin{proposition}
\label{charinvbis} Let $(\Lambda_1,\Lambda_2,\Lambda_3)$ be a dual $3$-net of order $n\geq 4$ realizing a group $G$. If every point of $(\Lambda_1,\Lambda_2,\Lambda_3)$ is the center of an involutory homology which preserves $(\Lambda_1,\Lambda_2,\Lambda_3)$, then   $(\Lambda_1,\Lambda_2,\Lambda_3)$ is triangular.
\end{proposition}
\begin{proof} {}From Proposition \ref{cltype} and Example \ref{exconicline}, $(\Lambda_1,\Lambda_2,\Lambda_3)$ is not of conic-line type. For $n=9$, $(\Lambda_1,\Lambda_2,\Lambda_3)$ does not lie on any non-singular cubic $\cF$ since no non-singular cubic has twenty-seven inflection points. Therefore the assertion follows
from Proposition \ref{charinv}.
\end{proof}
A useful generalization of Proposition \ref{charinvbis} is given in the proposition below.
\begin{proposition}
\label{28febbrbis} Let $(\Lambda_1,\Lambda_2,\Lambda_3)$ be a dual $3$-net of order $n\geq 4$ realizing a group $G$. Let $\cU$ be the set of all involutory homologies preserving $(\Lambda_1,\Lambda_2,\Lambda_3)$ whose centers are points of $(\Lambda_1,\Lambda_2,\Lambda_3)$. If $|\cU|\ge 3$ and $\cU$ contains two elements whose centers lie in different components, then the following assertions hold:
\begin{itemize}
\item[\rm(i)] every component contains the same number of points that are centers of involutory homologies in $\cU$.
\item[\rm(ii)] the points of $(\Lambda_1,\Lambda_2,\Lambda_3)$ which are centers of involutory homologies in $\cU$ form a triangular dual $3$-subnet $(\Gamma_1,\Gamma_2,\Gamma_3)$.
\item[\rm(iii)] Let $M$ be the cyclic subgroup associated to $(\Gamma_1,\Gamma_2,\Gamma_3)$.
Then either $(\Lambda_1,\Lambda_2,\Lambda_3)$ is also triangular, or
$$ |G|<\left\{
\begin{array}{lll}
\,\,\,|G:M|^2,\,\,\,{\rm{when\,\, gcd.}}(3,|G|)=1;\\
3|G:M|^2, \,\,\, {\rm{when\,\, gcd.}}(3,|G|)=3.
\end{array}
\right.
$$
\end{itemize}
\end{proposition}
\begin{proof} Let $\cG$ be the projectivity group preserving $(\Lambda_1,\Lambda_2,\Lambda_3)$. Let $(ijk)$ denote any permutation of $(123)$. As we have already observed in the proof of Proposition \ref{charinv}, if $\varphi\in\cG$ is an involutory homology with center $P\in \Lambda_i$, then $\varphi$ preserves $\Lambda_i$ and interchanges $\Lambda_j$ with $\Lambda_k$. If $\sigma\in\cG$ is another involutory homology with center $R\in\Lambda_j$ then $\sigma\varphi\sigma$ is also an involutory homology whose center $S$ is the common point of $\Lambda_k$ with the line $\ell$ through $P$ and $R$. In terms of dual $3$-subnets, this yields (i) and (ii). Let $m$ be the order of $(\Gamma_1,\Gamma_2,\Gamma_3)$. For $m=2$, $(\Gamma_1,\Gamma_2,\Gamma_3)$ is triangular. For $m=3$, $\Gamma_1\cup\Gamma_2\cup\Gamma_3$  is the Hesse configuration, and hence $(\Gamma_1,\Gamma_2,\Gamma_3)$ is triangular. This holds true for $m\ge 4$ by Proposition \ref{charinvbis} applied to $(\Gamma_1,\Gamma_2,\Gamma_3)$.

To prove (iii), assume that $(\Lambda_1,\Lambda_2,\Lambda_3)$ is not triangular and take a point $P$ from some component, say $\Lambda_3$, that does not lie on the sides of the triangle associated to
$(\Gamma_1,\Gamma_2,\Gamma_3)$. Since $(\Gamma_1,\Gamma_2,\Gamma_3)$ is triangular, it can play the role of $(\Lambda_1,\Lambda_2,\Lambda_3)$ in Section \ref{triangular3net}, and we use the notation introduced there. From the second assertion of Proposition \ref{trihomhom+}, the point has as many as $|\Theta|$ distinct images, all lying in  $\Lambda_3$. Therefore, $|G|=|\Lambda_3|>|\Theta|$. Using Proposition \ref{trihomhom+}, $|\Theta|$ can be written in function of $|M|$ giving the assertion
\end{proof}
Let $\cU_2$ be the set of all involutory homologies with center in $\Lambda_2$  which interchanges $\Lambda_1$ and $\Lambda_3$. There is a natural injective map $\Psi$ from $\cU_2$ to $G$ where $\Psi(\psi)=g$ holds if and only if the point $g_2\in \Lambda_2$ is the center of $\psi$.
\begin{proposition} \label{15aprile}
Let $(\Lambda_1,\Lambda_2,\Lambda_3)$ be a dual $3$-net of order $n\geq 4$ realizing a group $G$. If $|\cU_2|\geq 2$ then the following hold.
\begin{itemize}
\item[\rm(i)] $\cU_2$ is closed by conjugation, that is, $\psi\omega\psi\in \cU_2$ whenever $\psi,\omega\in \cU_2$.
\item[\rm(ii)] If $g,h\in \Psi(\cU_2)$ then $gh^{-1}g\in\Psi(\cU_2)$.
\item[\rm(iii)] If $G$ has a cyclic subgroup $H$ of order $6$ with $|H\cap \Psi(\cU_2)|\geq 3$ and $1\in H\cap \Psi(\cU_2)$, then either $\Psi(\cU_2)=H$, or $\Psi(\cU_2)$ is the subgroup of $H$ of order $3$.
\end{itemize}
\end{proposition}
\begin{proof} For $\psi,\omega \in \cU_2$, the conjugate $\tau=\psi\omega\psi$ of $\omega$ by $\psi$ is also an involutory homology. Let $g=\Psi(\psi)$ and $h=\Psi(\omega)$. Then  the center of $\tau$ is $\psi(h_2)$. For $x\in G$, the image of $x_1$ under $\tau$ is $y_3\in \Lambda_3$ with $y=xgh^{-1}g$. This shows that the center of $\tau$ is also in $\Lambda_2$; more precisely
\begin{equation}
\label{eq15aprile} \Psi(\psi\omega\psi)=\Psi(\psi)(\Psi(\omega))^{-1}\Psi(\psi).
\end{equation}
In the case where $G$ has a cyclic subgroup $H$ of order $6$, assume the existence of three distinct elements $\psi,\omega\,\rho \in \cU_2$ such that  $g=\Psi(\psi),\,h=\Psi(\omega)$, and $r=\Psi(\rho)$ with $g,h,r\in H$. Then $H$ contains $gh^{-1}g,\,hg^{-1}h,\,g^2$ and $h^2$. From this assertion (iii) follows.
\end{proof}

\section{Dual $3$-nets containing algebraic $3$-subnets of order $n$ with $n\geq 5$.}
A key result is the following proposition.
\begin{proposition}
\label{superlemmaszeged} Let $G$ be a group containing a proper abelian  subgroup $H$ of order $n\geq 5$. Assume that a dual $3$-net $(\Lambda_1,\Lambda_2,\Lambda_3)$ realizes $G$ such that all its dual $3$-subnets $(\Gamma_1^j,\Gamma_2,\Gamma_3^j)$ realizing $H$ as a subgroup of $G$ are algebraic. Let $\cF_j$ be the cubic  through the points of $(\Gamma_1^j,\Gamma_2,\Gamma_3^j)$. If $(\Lambda_1,\Lambda_2,\Lambda_3)$ is not algebraic then $\Gamma_2$ contains three collinear points and one of the following holds:
\begin{itemize}
\item[\rm(i)] $\Gamma_2$ is contained in a line.
\item[\rm(ii)] $n=5$ and there is an involutory homology with center in $\Gamma_2$ which preserves every $\cF_j$ and interchanges $\Lambda_1$ and $\Lambda_3$.
\item[\rm(iii)] $n=6$ and there are three involutory homologies with center in $\Gamma_2$ which preserves every $\cF_j$ and interchanges $\Lambda_1$ and $\Lambda_3$.
\item[\rm(iv)] $n=9$ and $\Gamma_2$ consists of the nine common inflection points of  $\cF_j$.
\end{itemize}
\end{proposition}
We need the following technical lemma.
\begin{lemma} \label{prop:abel}
Let $A=(A,\oplus)$, $B=(B,+)$ be abelian groups and consider the injective maps
$\alpha,\beta,\gamma:A\to B$ such that $\alpha(x)+\beta(y)+\gamma(z)=0$ if and only if
$z=x\oplus y$. Then, $\alpha(x)=\varphi(x)+a$, $\beta(x)=\varphi(x)+b$,
$\gamma(x)=-\varphi(x)-a-b$ for some injective homomorphism $\varphi:A\to B$ and elements
$a,b\in B$.
\end{lemma}
\begin{proof} Define $a=\alpha(0),\,b=\beta(0)$ and $\varphi(x)=-\gamma(x)-a-b$. For $x=0,\,z=y$, we obtain that
$\alpha(0)+\beta(y)+\gamma(y)=0$ whence $\beta(y)=-\gamma(y)-a=\varphi(y)+b$. Similarly, for $y=0,\,z=x$, we obtain that $\alpha(x)+\beta(0)+\gamma(x)=0$ whence $\alpha(x)=-\gamma(x)-b=\varphi(x)+a$. Finally, for any $x,y\in G$
$$
\begin{array}{lll}
\varphi(x)+\varphi(y)-\varphi(x+y)&=&\varphi(x)+a+\varphi(y)+b-(\varphi(x+y)+a+b)\\
&=& \alpha(x)+\beta(y)+\gamma(x+y)=0.
\end{array}
$$
Therefore, $\varphi:\,A\to B$ is a group homomorphism.
\end{proof}
Let $A=(A,\oplus)$  be an abelian group and $\alpha,\beta,\gamma$ injective maps from $A$ to $PG(2,\mathbb{K})$. The triple $(\alpha,\beta,\gamma)$ is a \emph{realization} of $A$ if the points $\alpha(x),\beta(y),\gamma(z)$ are collinear if and only if $z=x\oplus y$.
Since $(\Lambda_1,\Lambda_2,\Lambda_3)$ realizes $G$, the natural labeling gives rise to a realization $(\alpha,\beta,\gamma)$ such that $\alpha(G)=\Lambda_1,$ $\beta(G)=\Lambda_2,\gamma(G)=\Lambda_3$. Let $u\in G$. Since $H$ is a subgroup of $G$, the triple $$(\alpha_{u}(x)=\alpha(ux),\beta(y)=\beta(y),\gamma_{u}(z)=\alpha(uz))$$
provides a realization of $H$ such that
$$\alpha_{u}(H)=\Gamma_1^u,\,\beta(H)=\Gamma_2,\,\gamma_{u}(H)=\Gamma_3^{u}.$$
Therefore, Lemma \ref{prop:abel} has the following corollary where
$(\cF_j,*)$ denotes the additive groups of the plane cubic $\cF_j$ through the points of $(\Gamma_1^j,\Gamma_2,\Gamma_3^j)$ where,  for $u=1$, we write $(\cF,+),\alpha,\beta,\gamma,\Gamma_1,\Gamma_2,\Gamma_3$.
\begin{lemma}
\label{abeliansub} There exist two realizations from $H$ into $PG(2,\mathbb{K})$, say
$(\alpha,\beta,\gamma)$ and $(\alpha_j,\beta_j,\gamma_j)$ with $$\alpha(H)=\Gamma_1,\,\beta(H)=\Gamma_2,\,
\gamma(H)=\Gamma_3,\,\alpha_j(H)=\Gamma_1^j,\,\beta_j(H)=\Gamma_2,\,
\gamma_j(H)=\Gamma_3^j$$ such that
\begin{eqnarray*}
\alpha(x)=\varphi(x)+a,\beta(y)=\varphi(y)+b,\gamma(z)=\varphi(z)+c,\\
\alpha_j(x)=\varphi_j(x)*a_j,\beta_j(y)=\varphi_j(y)*b_j,\gamma_j(z)=\varphi_j(z)*c_j
\end{eqnarray*}
for every $x,y,z\in H$ where both $\varphi:\,H\to (\cF,+)$ and $\varphi_j:\,H\to (\cF_j,*)$ are injective homomorphisms, and $\varphi(y)+b=\varphi_j(y)*b_j$ for every $y\in H$.
\end{lemma}
To prove Proposition \ref{superlemmaszeged} we point out that $3b\in \varphi(H)$ if and only if $\Gamma_2$ contains three collinear points.
Suppose that $\varphi(x_1)+b,\varphi(x_2)+b,\varphi(x_3)+b$ are three collinear points. Then $\varphi(x_1)+b+\varphi(x_2)+b+\varphi(x_3)+b=0$ whence $\varphi(x_1+x_2+x_3)+3b=0$. Therefore $3b\in \varphi(H)$. Conversely, if $\varphi(t)=3b$, take three pairwise distinct elements $x_1,x_2,x_3 \in H$ such that $x_1+x_2+x_3+t=0$. Then $\varphi(x_1)+b+\varphi(x_2)+b+\varphi(x_3)+b=0$. Therefore, the points $\varphi(x_1)+b,\varphi(x_2)+b$ and $\varphi(x_3)+b$ of $\Gamma_2$ are collinear. Notice that the element $t=-x_1-x_2-x_3\in H$ is the same even if we make the computation with $\varphi_j$ and $b_j$.

We separately deal with two cases.
\subsection{$\Gamma_2$ contains no three collinear points}
By the preceding observation, $3b\not\in \varphi(H)$.
 For any $z\in H$ take four different elements $x_1,y_1,x_2,y_2$ in $H$ such that
 \begin{equation}
 \label{6apr}
 z=x_1\oplus y_1=x_2\oplus y_2.
 \end{equation}
 Then $\varphi(x_1)+b+\varphi(y_1)+b=\varphi(z)+2b=\varphi(x_2)+b+\varphi(y_2)+b$. Let $P_i=\beta(x_i),\,Q_i=\beta(y_i)$ for $i=1,2$. Then $P_i\neq Q_i$ and the lines $P_1Q_1$ and $P_2Q_2$ meet in a point $S$ in $\cF$ outside $\Gamma_2$. The same holds for $\cF_j$. Therefore each point $S$ is a common point of $\cF$ and $\cF_i$ other than those in $\Gamma_2$. As $S$ only depends on $z$ which can be freely choosen if $|H|\geq 4$, there are at least $n$ such points $S$. Hence, $\cF\cap \cF_j$ contains at least $2n\geq 10$ points. By B\'ezout's theorem either $\cF=\cF_j$, or they are reducible. We may assume that the latter case occurs. By Lemma \ref{20july}, we may assume that  both $(\Gamma_1,\Gamma_2,\Gamma_3)$ and $(\Gamma_1^j,\Gamma_2,\Gamma_3^j)$ are  of conic-line type.  Here  $\Gamma_2$ is contained in an irreducible conic $\cC$ which is a common component of $\cF$ and $\cF_j$. By Proposition \ref{cltype1}, $\cF=\cF_j$.

\subsection{$\Gamma_2$ contains  three collinear points}
This time, $3b\in \varphi(H).$ Let $\varphi(t)=3b$ with $t\in H.$
If either $\cF$ or $\cF_j$ is reducible,
then $\Gamma_2$ is contained in a line. Therefore, both $\cF$ and $\cF_j$ are assumed to be irreducible.

First, suppose in addition that $t\not\in 3H$.  For any $x\in H$,  let $y=2(\ominus x)\ominus t$. Observe that $y\neq x$. From
 \[2(\varphi(x)+b)+\varphi(y)+b = \varphi(t)+ \varphi(2x) +
\varphi(y) =0,\]
the point $Q=\beta(y)$ is the tangential point of $P=\beta(x)$ on $\cF$. Therefore,
 $\beta$ determines the tangents of $\cF$ at its points in $\Gamma_2$. This holds true for  $\cF_j$. From Lemma \ref{abeliansub}, $\cF$ and $\cF_j$ share the tangents at each of their common points in $\Gamma_2$. Therefore $|\cF\cap \cF_j|\geq 2n\geq 10$, and $\cF=\cF_j$ holds.

It remains to investigate the case where $3b=\varphi(3t_0)$ holds for some $t_0\in H$. Replacing $b$ by $b-\varphi(t_0)$ shows that $3b=0$ may be assumed. Therefore, the point $P=\varphi(y)+b$ with $y\in H$ is an inflection point of $\cF$ if and only if $3y=0$. Furthermore, if $3y\neq 0$ then $Q=\varphi(\ominus (2y))+b$ is the tangential point of $P$ on $\cF$. Therefore, $\beta$ determines the tangents of $\cF$ at its points in $\Gamma_2$. The same holds true for $\cF_j$. By Lemma \ref{abeliansub}, $P=\beta(y)$ is an inflection point of both $\cF$ and $\cF_j$ or none of them. In the latter case, $\cF$ and $\cF_j$ have the same tangent at $P$.

Let $m$ be the number of common inflection points of $\cF$ and $\cF_j$ lying in $\Gamma_2$. Obviously, $P=\varphi(0)+b$ is such a point, and hence $m\geq 1$. On the other hand, $m$ may assume only three values, namely $1$, $3$ and $9$. If $m=9$, then $\cF$ is non-singular and  $\Gamma_2$ consists of all the nine inflection points of $\cF$. The same holds for $\cF_j$. If $m=3$ then $\cF$ and $\cF_j$ share their tangents at $n-3$ common points. Therefore, $2n-3\leq 9$ whence $n\le 6$.

If $n=6$ there are three common inflection points of $\cF$ and $\cF_j$, and they are collinear.  Let $H$ be the additive group of integers modulo $6$. Then the inflection points of $\cF$ lying on $\Gamma_2$ are $P_i=\varphi(i)+b$ with $i=0,2,4$ while the tangential point of $P_i=\varphi(i)+b$ with $i=1,3,5$ is $P_{-2i}=\varphi(-2i)+b$. Now
fix a projective frame with homogeneous coordinates $(X,Y,Z)$ in such a way that $$
\begin{array}{lll}
P_0=(1,0,1),\,P_1=(0,0,1),\,P_2=(0,1,1),\\
P_3=(0,1,0),\,P_4=(-1,1,0),\,P_5=(1,0,0).
\end{array}
$$
 A straightforward computation shows that $\cF_j$  is in  the pencil $\cP$ comprising the cubics $\cG_\lambda$ of equation
$$(X-Z)(Y-Z)(X+Y)+\lambda XYZ=0,\quad \lambda\in {\mathbb K},$$
with the cubic $\cG_\infty$ of equation $XYZ=0$. The intersection divisor of the plane cubics in $\cP$ is $P_0+P_2+P_4+2P_1+2P_3+2P_5$.
Moreover, the points $P_0,P_2,P_4$ are inflection points of all irreducible cubics in $\cP$, and
$$
\begin{array}{lll}
\psi_0:\,(X,Y,Z)\to(Z,-Y,X),\\
\psi_2:\,(X,Y,Z)\to(-X,Z,Y),\\
\psi_4:\,(X,Y,Z)\to(Y,X,Z),
\end{array}
$$
are the involutory homologies preserving every cubic in $\cP$, the center of $\psi_i$ being $P_i$, for $i=0,2,4$.

If $n=5$, the zero of $H$ is the only element $y$ with $3y=0$. This shows that $\cF$ (and $\cF_j$) has only one inflection point $P_0$ in $\Gamma_2$ and $P_0$ is not the tangential point of another point in $\Gamma_2$. Each of the remaining four points is the the tangential point of exactly one point in $\Gamma_2$. These four points may be viewed as the vertices of a quadrangle $P_1P_2P_3P_4$ such that the side $P_iP_{i+1}$ is tangent to $\cF$ at $P_i$ for every $i$ with $P_5=P_1$. Therefore the intersection divisor of $\cF$ and $\cF_j$ is $P_0+2P_1+2P_2+2P_3+2P_4$, and  $\cF_j$ is contained in a pencil $\cP$.

Fix a projective frame with homogeneous coordinates $(X,Y,Z)$ in such a way that $$P_1=(0,0,1),\,P_2=(1,0,0),\,P_3=(1,1,1),\,P_4=(0,1,0).$$
Then $P_0=(1,1,0)$.
The pencil $\cP$ is generated by the cubics $\cG$ and $\cD$ with equations $Y(X-Z)Z=0$ and $X(Y-X)(Y-Z)=0$, respectively. Therefore it consists of cubics $G_{\lambda}$ with equation
$$Y^2X-X^2Y+(\lambda-1)XYZ+X^2Z-\lambda YZ^2=0,$$
together with $\cG=\cG_{\infty}$.
Since the line $Z=0$ contains three distinct base points of the pencil, $P_0$ is a non-singular point of $\cG_{\lambda}$ for every $\lambda\in {\mathbb K}$, the tangent $\ell_{\lambda}$ to $\cG_{\lambda}$ at $P_0$ has equation
$-X+Y+\lambda Z=0$. Assume that $Q_0$ is an inflection point of $\cG_{\lambda}$. Then  $\ell_{\lambda}$ contains no point $P=(X,Y,1)$ from $\cG_{\lambda}$, that is, the polynomials $Y^2X-X^2Y+(\lambda-1)XY+X^2-\lambda Y=0$ and $-X+Y+\lambda=0$ have no common solutions. On the other hand, eliminating $Y$ from these polynomials gives $\lambda^2$. This shows that $Q_0$ is an inflection point for every irreducible cubic in $\cP$. Hence $P_0=Q_0$.
Therefore the involutory homology  $$\varphi:\,(X,Y,Z)\mapsto (-Y+Z,-X+Z,Z)$$ with center $P_0$
preserves each cubic in $\cP$.

This completes the proof of Proposition \ref{superlemmaszeged}.

In the case where $H$ is an abelian normal subgroup of $G$, we have the following result.
\begin{proposition}
\label{superlemma} Let $G$ be a group containing a proper abelian normal subgroup $H$ of order $n\geq 5$. If a dual $3$-net $(\Lambda_1,\Lambda_2,\Lambda_3)$ realizes $G$ such that all its dual $3$-subnets realizing $H$ as a subgroup of $G$ are algebraic, then either {\rm{(I)}} or {\rm{(II)}} of Theorem \ref{mainteo}  holds.

\end{proposition}
\begin{proof} The essential tool in the proof is Proposition \ref{superlemmaszeged}. Assume on the contrary that neither (I) nor (II) occurs.

If every $H$-member is contained in a line then every dual $3$-net realizing $H$ as a subgroup of $G$ is triangular. From Proposition \ref{cortetra0}, either (I) of (II) follows.

Take a $H$-member not contained in a line. Since $H$ is a normal subgroup, that $H$-member can play the role of $\Gamma_2$ in Proposition \ref{superlemmaszeged}. Therefore, one of the three sporadic cases in Proposition \ref{superlemmaszeged} holds. Furthermore, from the proof of that proposition, every $\cF_j$ is irreducible, and hence
neither $\Gamma_1^j$ nor $\Gamma_3^j$ is contained in a line. Therefore, no $H$-member is contained in a line. Since $H$ is a normal subgroup, every $3$-subnet $(\Gamma_1^i,\Gamma_2^j,\Gamma_3^s)$  realizing $H$ as a subgroup of $G$ lies in an irreducible plane cubic $\cF(i,j)$.

Therefore we can assume that all $H$-members have the exceptional configurations described in (ii), (iii) or (ivc) of Proposition \ref{superlemmaszeged}. We separately deal with the cases $n=5,6$ and $9$.
\subsection{$n=9$}
{}From (iv) of Proposition \ref{superlemmaszeged}, the cubics $\cF_j$ share their nine inflection points which form  $\Gamma_2$. So it is possible to avoid this case by replacing $\Gamma_2$ with $\Gamma_1$ so that $\Gamma_2$ will not have any inflection point of $\cF$.
\subsection{$n=6$}
Every $H$-member $\Gamma_2$ contains three collinear points, say $Q_1,Q_2,Q_3$, so that $Q_r$ is the center of an involutory homology $\psi_r$ interchanging $\Lambda_1$ and $\Lambda_3$. Relabeling the points of the dual $3$-net permits us to assume that $Q_1=1_2$. Then for all $x \in G$, $\psi_1$ interchanges the points $x_1$ and $x_3$. The point $a_2\in \Lambda_2$ is the intersection of the lines $y_1(ya)_3$, with $y\in G$. These lines are mapped to the lines $(ya)_1y_3$, which all contain the point $(a^{-1})_2$ of $\Lambda_2$. Therefore, the involutory homology $\psi_1$ leaves $\Lambda_2$ invariant. This holds true for all involutory homologies with center in $\Lambda_1 \cup \Lambda_2 \cup \Lambda_3$. Since the $H$-members partition each component of $(\Lambda_1,\Lambda_2,\Lambda_3)$ and every $H$-member comprises six points, it turns out that half of the points of $(\Lambda_1,\Lambda_2,\Lambda_3)$ are the centers of involutory homologies preserving $(\Lambda_1,\Lambda_2,\Lambda_3)$. Therefore, Proposition \ref{28febbrbis}(iii) applies. As in  Proposition \ref{28febbrbis}, let $M$ denote the subgroup of $G$ such that the dual $3$-subnet consisting of the centers of involutory homologies realizes $M$. As $|G:M|=2$, \ref{28febbrbis}(iii) implies $|G|<6$, a contradiction.
\subsection{$n=5$}
The arguments in discussing case $n=6$ can adapted for case $n=5$. This time, Proposition \ref{superlemmaszeged} gives $|G:M|=5$. By Proposition \ref{28febbrbis}(iii), if $G$ contains an element of order $3$ then $|G|<75$, otherwise $|G|<25$. In the former case, the element of order $3$ of $G$ is in $C_G(H)$, hence $G$ contains a cyclic normal subgroup of order $15$. Then, $(\Lambda_1,\Lambda_2,\Lambda_3)$ is algebraic by Proposition \ref{superlemmaszeged}. If $G$ has no element of order three then $|G|<25$ and $G$ contains a normal subgroup of order $10$ which is either cyclic or dihedral. By Propositions \ref{cortetra} and \ref{superlemmaszeged} either {\rm{(I)}} or {\rm{(II)}} of Theorem \ref{mainteo}  holds.
\end{proof}
A corollary of Proposition \ref{superlemma} is  the following result.
\begin{theorem}
\label{teo1} Every dual $3$-net $(\Lambda_1,\Lambda_2,\Lambda_3)$ realizing an abelian group $G$ is algebraic.
\end{theorem}
\begin{proof} By absurd, let $n$ be the smallest integer for which a counter-example $(\Lambda_1,\Lambda_2,\Lambda_3)$ to Theorem \ref{teo1} exists. Since any dual $3$-net of order $\leq 8$ is algebraic by Propositions \ref{clac8szeged} and \ref{th5.3}, we have that $n\geq 9$.  Furthermore, again by Proposition \ref{th5.3}, $G$ has composite order. Since $n$ is chosen to be as small as possible, from Proposition \ref{superlemma}, $|G|$ has only one prime divisor, namely either $2$ or $3$. Since $|G|\geq 9$, either $|G|=2^r$ with $r\geq 4$, or $|G|=3^r$ with $r\geq 2$. In the former case, $G$ has a subgroup $M$ of order $8$, and every dual $3$-subnet realizing $M$ is algebraic, by Proposition \ref{clac8szeged}. But, this together with Proposition \ref{superlemma} show that $(\Lambda_1,\Lambda_2,\Lambda_3)$ is not a counter-example. In the latter case $G$ contains no element of order $9$ and hence it is an elementary abelian group.
But then  $(\Lambda_1,\Lambda_2,\Lambda_3)$ is algebraic by Proposition \ref{clac9szeged}.
\end{proof}

\section{Dual $3$-nets realizing $2$-groups}
\begin{proposition}
\label{cl2group} Let $G$ be a group of order $n=2^h$ with $h\geq 2$. If  $G$ can be realized by a dual $3$-net $(\Lambda_1,\Lambda_2,\Lambda_3)$ then one of the following holds.
\begin{itemize}
\item[\rm{(i)}] $G$ is cyclic.
\item[\rm{(ii)}] $G\cong C_m\times C_k$ with $n=mk$.
\item[\rm{(iii)}] $G$ is a dihedral.
\item[\rm{(iv)}] $G$ is the quaternion group of order $8$.
\end{itemize}
\end{proposition}
\begin{proof} For $n=4,8$, the classification follows from Propositions \ref{clac8szeged}, \ref{tetradiedral} and \cite[Theorem 4.2]{ys2004}. Up to isomorphisms, there exist fourteen groups of order $16$; each has a subgroup $H$ of index $2$ that is either an abelian or a dihedral group. In the latter case, $G$ is itself dihedral, by Proposition \ref{cortetra}. So, Proposition \ref{superlemma} applies to $G$ and $H$ yielding that $G$ is abelian. This completes the proof for $n=16$. By induction on $h$ we assume that Proposition \ref{cl2group} holds for $n=2^h\geq 16$ and we are going to show that this remains true for $2^{h+1}$. Let $H$ be a subgroup of $G$ of index $2$. Then $|H|=2^h$ and one of the cases (i), (ii), and (iii) hold for $H$. Therefore, the assertion follows from Propositions \ref{superlemma} and \ref{cortetra}.
\end{proof}

\section{Dual $3$-nets containing algebraic $3$-subnets of order $n$ with $2\leq n \leq 4$.}
It is useful to investigate separately two cases according as $n=3,4$ or $n=2$. An essential tool in the investigation is $M=\cC_G(H),$ the centralizer of $H$ in $G.$
\begin{proposition}
\label{7aprileszeged4} Let $G$ be a finite group containing a normal subgroup $H$ of order $n$ with $n=3$ or $n=4$. Then every dual $3$-net $(\Lambda_1,\Lambda_2,\Lambda_3)$ realizing $G$ is either algebraic or of tetrahedron type, or, $G$ is isomorphic either to the quaternion group of order $8$, or to ${\rm{Alt}}_4$, or to ${\rm{Sym}}_4$.
\end{proposition}
\begin{proof} First we investigate the case where $M>H$. Take an element $m\in M$ outside $H$. Then the subgroup $T$ of $G$ generated by $m$ and $H$ is abelian, and larger than $H$. Since $|H|\geq 3$, then $|T|\geq 6$. If all $H$-members of $(\Lambda_1,\Lambda_2,\Lambda_3)$ are contained in a line then $(\Lambda_1,\Lambda_2,\Lambda_3)$ is either triangular or of tetrahedron type by Proposition \ref{cortetra0}. Assume that $\Gamma_2$ is an $H$-member which is not contained in a line. Let $\Gamma_2'$ be the $T$-member containing $\Gamma_2$. We claim that $(\Lambda_1,\Lambda_2,\Lambda_3)$ is algebraic. If not then one of the exceptional cases (iii) or (iv) of Proposition \ref{superlemmaszeged} must hold. Clearly, in these cases $|H|=3$. However, the centers of the involutory homologies mentioned in Proposition \ref{superlemmaszeged} correspond to the points in the $H$-member $\Gamma_2$. As these centers must be collinear, we obtain that $\Gamma_2$ is contained in a line, a contradiction.

Assume that $M=H$. Then $G/H$ is an automorphism group $H$. If $H$ is $C_3$ or $C_4$ then $|\mathrm{Aut}(H)|=2$ and $G$ is either a dihedral group or the quaternion group of order $8$. If $H\cong C_2\times C_2$, then $G$ is a subgroup of  $\rm{Sym}_4$.  The possibilities for $G$
other than $H$ and the dihedral group of order $8$ are two, either $\rm{Alt}_4$, or $\rm{Sym}_4$. Since all these groups are allowed in the proposition, the proof is finished.
\end{proof}
\begin{proposition}
\label{lemcen} Let $G$ be a finite group with a central involution which contains no normal subgroup $H$ of order $4$. Then a dual $3$-net $(\Lambda_1,\Lambda_2,\Lambda_3)$ realizing $G$  is either algebraic or of tetrahedron type.
\end{proposition}
\begin{proof} Let $H$ be the normal subgroup  generated by the (unique) central involution of $G$. Two cases are separately investigated according as a minimal normal subgroup $\bar{N}$ of the factor group $\bar{G}=G/H$ is solvable or not. Let $\sigma$ be the natural homomorphism $G\to \bar{G}$. Let $N=\sigma^{-1}(\bar{N})$.

If $\bar{N}$ is solvable, then  $\bar{N}$ is an elementary abelian group of order $d^h$ for a prime $d$. Furthermore, $N$ is a normal subgroup of $G$ and $\bar{N}=N/H$. If $N$ is abelian, then $|N|\geq 6$ and the assertion follows from Proposition \ref{superlemma} and Theorem \ref{teo1}.

Bearing this in mind, the case where $d=2$ is investigated first. Then $N$ has order $2^{h+1}$ and is a normal subgroup of $G$. From Proposition \ref{cl2group}, $N$ is either abelian or it is the quaternion group $Q_8$ of order $8$. We may assume that $N\cong Q_8$. From Proposition \ref{cl2group}, $N$ is not contained in a larger $2$-subgroup of $G$. Therefore $N$ is a (normal) Sylow 2-subgroup of $G$. We may assume that $G$ is larger than $N$. If $M=C_G(N)$ is also larger than $N$, take an element $t\in M$ of outside $N$. Then $t$ has odd order $\geq 3$. The group $T$ generated by $N$ and $t$ has order $8m$ and its subgroup $D$ generated by $t$ together with an element of $N$ of order $4$ is a (normal) cyclic subgroup of $M$  of order $4m$. But this contradicts Proposition  \ref{superlemma}, as $T$ is neither abelian nor dihedral. Therefore $M=N$, and hence $G/N$ is isomorphic to a subgroup $L$ of the automorphism group ${\mathrm{Aut}}(Q_8)$. Hence $|G|/|N|$ divides $24$. On the other hand, since $N$ is a Sylow $2$-subgroup of $G$, $|G/N|$ must be odd. Therefore $|G|=24$. Two possibilities arise according as either $G\cong SL(2,3)$ or $G$ is the dicyclic group of order $24$. The latter case cannot actually occur by Proposition \ref{superlemma} as the dicyclic group of order $24$ has a (normal) cyclic subgroup of order $12$.

To rule the case $G\cong SL(2,3)$ out we relay on Proposition \ref{superlemmaszeged} and \ref{15aprile} since
$SL(2,3)$ has four cyclic groups of order $6$. For this purpose, we show that every point in $\Lambda_2$ is the center of an involutory homology preserving $(\Lambda_1,\Lambda_2,\Lambda_3)$ whence the assertion will follow from Proposition \ref{charinv} applied to $\Lambda_2$. With the notation in Section \ref{seccolli}, (iii) Proposition \ref{superlemmaszeged} yields that $|\cU_2|\geq 3$. With the notation introduced in the proof of Proposition \ref{th5.3}, we may assume that the point $1_2$ is the center of an involutory homology $\epsilon$ in $\cU_2$. From (iii) of Proposition \ref{15aprile} every (cyclic) subgroup of $G$ of order $6$ provides (at least) two involutory homologies other than $\epsilon$. Therefore, $|\cU_2|\geq 9$, and every point $u_2\in \Lambda_2$ such that $u^3=1$ is the center of an involutory homology in $\cU$. A straightforward computation shows that every element in $G$ other than the unique involution $e$ can be written as $gh^{-1}g$ with $g^3=h^3=1$. Thus $|\cU|\geq 23$. The involutory homology with center $1_2$ cannot actually been an exception. To show this, take an element $g\in G$ of order $4$. Then $g_2$ is the center of an element in $\cU$. Since $1=g^2=g\cdot 1\cdot g$, this holds true for $1_2$. Therefore $|\cU|=24$. From (i) of Proposition \ref{15aprile}, $\cU$ also preserves $\Lambda_2$. This completes the proof.

Now, the case of odd $d$ is investigated. Since $|H|=2$ and $d$ are coprime, Zassenhaus' theorem \cite[10.1 Hauptsatz]{huppertI1967} ensures a complement $W\cong \bar{N}$ such that $N=W\ltimes H=W\times H$.  Obviously, $W$ is an abelian normal subgroup
of $G$ of order at least $3$. The assertion follows  from Propositions \ref{superlemma} and \ref{7aprileszeged4}.

If $\bar{N}$ is not solvable then it has a non-abelian simple group $\bar{T}$. Let $\bar{S}_2$ be a Sylow $2$-subgroup of $\bar{T}$. By Proposition \ref{cl2group}, the realizable $2$-group $S_2$ is either cyclic, product of two cyclic groups, dihedral or quaternion of order $8$. Thus, $\bar{S}_2$ is either cyclic, product of two cyclic groups or dihedral. As $\bar T$ is simple, $\bar{S}_2$ cannot be cyclic. In the remaining cases we can use the classification of finite simple groups of $2$-rank $2$ to deduce that $\bar T$ is either $\bar{T}\cong PSL(2,q^h)$ with an odd prime $q$ and $q^h\geq 5$, or $\bar{T}\cong\rm{Alt}_7$, cf. the Gorenstein-Walter theorem \cite{gorenstein1965}.

If $H\not\leq T'$ then $T=H\times T'$. As $T'\cong \bar T$, $T'$ contains an elementary abelian subgroup of order $4$, $G$ contains an elementary abelian group of order $8$, a contradiction. Therefore, $T$ is a central extension of either $PSL(2,q^h)$, with $q^h$ as before, or $\rm{Alt}_7$, with a cyclic group of order $2$. From a classical result of Schur \cite[Chapter 33]{aschbacher1993}, either $T\cong SL(2,q)$ or $T$ is the unique central extension of $\rm{Alt}_7$ with a cyclic group of order $2$. In the latter case, no dual $3$-net can actually realize $T$ since Proposition \ref{cl2group} applies, a Sylow $2$-subgroup of $T$ being isomorphic to a generalized quaternion group of order $16$. To finish the proof it suffices to observe that $SL(2,q^h)$, with $q^h$ as before, contains $SL(2,3)$ whereas no dual $3$-net can realize $SL(2,3)$ as we have already pointed out.
\end{proof}

\section{$3$-nets and non-abelian simple groups}
\begin{proposition}
\label{egysz} If a dual $3$-net realizes a non-abelian simple group $G$ then $G\cong \rm{Alt}_5$.
\end{proposition}
\begin{proof} Let $G$ be a non-abelian simple group, and consider a Sylow $2$-subgroup $S_2$ of $G$. From Proposition \ref{cl2group}, $S_2$ is dihedral since no Sylow $2$-subgroup of a non-abelian simple group is either cyclic or the direct product of cyclic groups, or the quaternion group of order $8$ [Reference???]. From the Gorenstein-Walter theorem \cite{gorenstein1965}, either $G\cong PSL(2,q^h)$ with an odd prime $q$ and $q^h\geq 5$, or $G\cong\rm{Alt}_7$. In the former case, $G$ has a subgroup $T$ of order $q^h(q^h-1)/2$ containing a normal subgroup of order $q^h$. Here $T$ is not abelian and is dihedral only for $q^h=5$. Therefore, Theorem \ref{teo1} and Proposition \ref{superlemma} leave only one case, namely $q=5$.  This also shows that  ${\rm{Alt}}_7$ cannot occur since ${\rm{Alt}}_7$ contains $PSL(2,7)$.
\end{proof}
Notice that by Proposition \ref{clac12}, computer results show that if $p=0$ then $\rm{Alt}_4$ cannot be realized in $PG(2,\mathbb K)$. This implies that that no dual $3$-net can realize $\rm{Alt}_5$.

\section{The proof of Theorem \ref{mainteo}}
Take a minimal normal subgroup $H$ of $G$. If $H$ is not solvable then $H$ is either a simple group or the product of isomorphic simple groups. From Proposition \ref{noelem}, the latter case cannot actually occur as every simple group contains an elementary abelian subgroup of order $4$.
Therefore, if $H$ is not solvable, $H\cong {\rm{Alt}}_5$ may be assumed by Proposition \ref{egysz}. Two cases are considered separately according as the centralizer $C_G(H)$ of $H$ in $G$ is trivial or not. If $|C_G(H)|>1$, take a non-trivial element $u\in C_G(H)$ and define $U$ to be the subgroup of $G$ generated by $u$ together with a dihedral subgroup $D_5$ of $H$ of order $10$. Since $u$ centralizes $D_5$, the latter subgroup is a normal subgroup of $U$. Hence $D_5$ a normal dihedral subgroup of $U$. From Proposition \ref{cortetra}, $M$ itself must be dihedral. Since the center of a dihedral group has order $2$, this implies that $u$ is an involution. Now, the subgroup generated by $u$ together with an elementary abelian subgroup of $H$ of order $4$ generate an elementary abelian subgroup of order $8$. But this contradicts Proposition \ref{noelem}. Therefore, $C_G(H)$ is trivial, equivalently, $G$ is contained in the automorphism group of $H$. From this, either $G=H$ or $G\cong PGL(2,5)$. In the latter case, $G$ contains a subgroup isomorphic to the semidirect product of $C_5$ by $C_4$. But this contradicts Proposition \ref{cortetra}. Therefore, if $H$ is not solvable then $H\cong \rm{Alt}_5$.

If $H$ is solvable then it is an elementary abelian group of order $d\geq 2$.
If $d$ is a power of $2$ then $d=2$ or $d=4$ and Theorem \ref{mainteo} follows from Propositions \ref{cl2group} and \ref{lemcen}. If  $d$ is a power of an odd prime, Theorem \ref{mainteo} is obtained by Propositions \ref{superlemma} and  \ref{7aprileszeged4}. \qed

\vspace{0,5cm}\noindent {\em Authors' addresses}:

\vspace{0.2cm}\noindent G\'abor KORCHM\'AROS\\ Dipartimento di
Matematica e Informatica\\ Universit\`a della Basilicata\\ Contrada Macchia
Romana\\ 85100 Potenza (Italy).\\E--mail: {\tt
gabor.korchmaros@unibas.it }

\vspace{0.2cm}\noindent G\'abor P\'eter NAGY\\ Bolyai Institute \\
University of Szeged \\ Aradi v\'ertan\'uk tere 1\\
H-6720 Szeged (Hungary).\\
E--mail: \texttt{nagyg@math.u-szeged.hu}

\vspace{0.2cm}\noindent Nicola PACE\\ Department of Mathematical Sciences  \\
Florida Atlantic University \\
777 Glades Road \\
Boca Raton, FL 33431, USA}
\end{document}